\documentclass[12pt, draft]{amsart}
\usepackage{latex_base}
\usepackage{macros}

\author{Janin Heuer}
\address{Janin Heuer, Technische Universit\"at Braunschweig, Institut f\"ur Analysis und Algebra, AG Algebra, Universit\"atsplatz 2, 38106 Braunschweig,
 Germany\medskip}
\email{janin.heuer@tu-braunschweig.de}
\author{Timo de Wolff}
\address{Timo de Wolff, Technische Universit\"at Braunschweig, Institut f\"ur Analysis und Algebra, AG Algebra, Universit\"atsplatz 2, 38106 Braunschweig,
 Germany\medskip}
\email{t.de-wolff@tu-braunschweig.de}



\title[]{The Duality of SONC: Advances in Circuit-based Certificates}

\begin{document}

\begin{abstract}
	The cone of sums of nonnegative circuits (SONCs) is a subset of the cone of nonnegative polynomials / exponential sums, which has been studied extensively in recent years. 
	In this article, we construct a subset of the SONC cone which we call the DSONC cone. 
	The DSONC cone can be seen as an extension of the dual SONC cone; 
	 membership can be tested via linear programming. 
	We show that the DSONC cone is a proper, full-dimensional cone, we provide a description of its extreme rays, and collect several properties that parallel those of the SONC cone.
	Moreover, we show that functions in the DSONC cone cannot have real zeros, which yields that DSONC cone does not intersect the boundary of the SONC cone.
	Furthermore, we discuss the intersection of the DSONC cone with the SOS and SDSOS cones.
	Finally, we show that circuit functions in the boundary of the DSONC cone are determined by points of equilibria, which hence are the analogues to singular points in the primal SONC cone, and relate the DSONC cone to tropical geometry.
\end{abstract}

\maketitle

\section{Introduction}

For a finite set $\struc{A} \subset \R^n$ we define the space $\struc{\R^A}$ of \struc{exponential sums} supported on $A$ as the set of functions of the form 
\begin{align}
	f(\xb) \ = \ \sum_{\alpb \in A} c_{\alpb} \expalpha,
\label{eqn:ExpSum}
\end{align}
where $c_{\alpb} \in \R$ for all $\alpb \in A$.
Exponential sums, which are also referred to as \struc{signomials}, are a class of functions relevant to a variety of applications.
Specifically, one is interested in minimizing functions of the form \cref{eqn:ExpSum} (potentially subject to constraints); see e.g. \cite{Erguer:Paouris:Rojas:TropVarietiesExpSums,Boyd:et:al:TutorialOnGP,Duffin:Peterson:Signomials} for an overview.

In the case $A \subset \N^n$, the function $f$ defined in \cref{eqn:ExpSum} can be expressed as a polynomial on the positive orthant via the logarithmic change of variables
\begin{align*}
	(y_1, \ldots, y_n) \ \mapsto \ f(\ln(y_1), \ldots, \ln(y_n)) \ = \ \sum_{\alpb \in A} c_{\alpb} y_1^{\alpb_1} \cdots y_n^{\alpb_n}.
\end{align*}
The problem of minimizing exponential sums or polynomials is closely related to the question of deciding whether these functions are nonnegative. 
I.e., instead of minimizing an exponential sum or polynomial $f$ we can equivalently search for the maximal constant $\gamma \in \R$ such that $f - \gamma \ge 0$.
It is well-known that these nonnegativity problems are extremely challenging, even in the polynomial case with low degrees, as this class contains (when phrased as a decision problem) for example NP-complete problems like \textsc{MAXCUT}; see e.g. \cite{Laurent:Survey}.

A common way to tackle these problems is to search for \struc{certificates of nonnegativity}.
These are conditions on exponential sums or polynomials which imply nonnegativity while being easier to test than nonnegativity itself.
The classical approach, dating back to the 19th century (e.g., \cite{Hilbert:Seminal}), is to represent a polynomial as a \struc{sum of squares (SOS)} of other polynomials.
Computationally, this approach requires \struc{semidefinite programming (SDP)}, and has been studied extensively in the last two decades.
For an overview see e.g. \cite{Lasserre:GlobalOpt,Parrilo:Thesis}.
Although being extremely successful, the SOS method neither preserves sparsity of the support set $A$, nor can it be applied to the case of exponential sums with general support sets in $\R^n$.
Moreover, it comes with significant computational challenges, as for $n$-variate polynomials in degree $2d$ the corresponding SDP contains matrices of size $\binom{n+d}{d}$.

For this reason, we focus on an alternative certificate of nonnegativity which attempts to decompose exponential sums into \struc{sums of nonnegative circuits (SONC)}.
Circuit polynomials were studied in a special case by Reznick in \cite{Reznick:AGI} (where they are called \struc{agiforms}), and then introduced for the general polynomial case by Iliman and the second author in \cite{Iliman:deWolff:Circuits}.
Circuit polynomials are supported on minimally affine dependent sets $A$ and satisfy a collection of further conditions; for a formal introduction see \cref{definition:Circuit}.
Iliman and the second author showed that deciding nonnegativity of circuit polynomials simply requires solving a system of linear equations, which is a consequence of the \struc{inequality of arithmetic and geometric means (AM/GM)}.
% % % and does not depend on the degree of the polynomial.
These results can immediately be adapted from the polynomial case to the case of exponential sums; see \cref{theorem:NonnegativeCircuits}. 

The set $\struc{\ASONC}$ of SONCs supported on $A$ is a closed convex cone with the same dimension as the cone of nonnegative polynomials supported on $A$ \cite{Dressler:Iliman:deWolff:Positivstellensatz}, which was studied by various authors in different settings.
Chandrasekaran and Shah introduced AM/GM-based methods under the name \struc{sums of AM/GM exponentials (SAGE)} in \cite{Chandrasekaran:Shah:SAGE-REP}, which were further investigated in, e.g., \cite{Murray:Chandrasekaran:Wierman:NewtonPolytopes,Murray:Chandrasekaran:Wierman:SigOptREP}.
Another closely related approach is the \struc{$\cS$-cone} introduced by Katth\"an, Theobald and Naumann in \cite{Katthaen:Naumann:Theobald:UnifiedFramework}.
Furthermore, Forsg\aa rd and the second author provided a new, more abstract interpretation of the SONC approach in \cite{Forsgaard:deWolff:BoundarySONCcone}.

Recently, the \emph{dual} cone $\struc{\dualASONC}$ corresponding to these objects has been studied in, e.g., \cite{Dressler:Heuer:Naumann:deWolff:DualSONCLP,Dressler:Naumann:Theobald:DualSONC,Katthaen:Naumann:Theobald:UnifiedFramework}.
This cone is particularly interesting since first, circuit functions whose coefficients are contained in the dual SONC cone are also contained in the primal SONC cone; see \cite[Proposition 3.6]{Dressler:Heuer:Naumann:deWolff:DualSONCLP}, and second, optimizing over the dual SONC cone can be done via linear programming; see \cite[Proposition 4.1]{Dressler:Heuer:Naumann:deWolff:DualSONCLP}.
Since optimizing over the primal SONC cone requires geometric or relative entropy programming, see e.g. \cite{Chandrasekaran:Shah:SAGE-REP, Dressler:Iliman:deWolff:FirstConstrained, Iliman:deWolff:FirstGP}, relaxing the SONC certificate to a certificate of nonnegativity induced by the dual SONC cone provides a significant computational advantage.

\medskip

In this article we extend and generalize the observations made in \cite{Dressler:Heuer:Naumann:deWolff:DualSONCLP} by defining the \struc{DSONC cone} as the conic closure of the set of circuit functions with coefficient vector in the dual SONC cone; see \cref{definition:DSONC} for a formal definition.
On the one hand, the DSONC cone $\struc{\ADSONC}$ is a subset of the primal SONC cone $\ASONC$ (\cref{corollary:DSONCinPrimal}), which contains more functions than the certificates introduced in \cite{Dressler:Heuer:Naumann:deWolff:DualSONCLP}.
On the other hand, as for the dual SONC cone, one can check membership in the DSONC cone by linear programming, which makes it a very interesting object with high potential for applications.

\medskip

Our contribution consists of a fundamental analysis of the DSONC cone.

In \cref{section:DSONCIntroduction}, we formally introduce the DSONC cone in \cref{definition:DSONC}. 
We show that the statement ``SONC = SAGE'' (i.e., results showing that approaches using circuit functions or AM/GM exponentials lead to the same convex cone) translate to the DSONC case, leading to a ``DSONC = DSAGE'' counterpart result, \cref{corollary:DSONCisDSAGE}.

In \cref{section:PropertiesDSONCCone} we discuss further structural properties of DSONCs. 
In \cref{proposition:DSONCconeProper} we prove that the DSONC cone is a proper, full-dimensional cone in $\R^A$, and that membership of a single circuit function $f$ in the DSONC cone can be decided by an invariant, which we call \struc{dual circuit number}, an analogue of the circuit number, which decides membership of a circuit function in the SONC cone; see \cref{definition:DualCircuitNumber} and \cref{corollary:DualCircuitNumber}.
Moreover, the dual circuit number proves a close relation of the DSONC cone to an object called the power cone; see \cref{equation:PowerCone}.
As in the primal case, the dual circuit number depends on the coefficients of $f$ and barycentric coordinates given by the supporting circuit alone.

The dual and the primal circuit number are closely related to each other, \cref{corollary:PrimalDualConnection}, which allows to improve bounds computed via LPs exploiting DSONC certificates further; see \cref{remark:LinearProgramming}.
Furthermore, we derive from this relation that the DSONC cone is contained in the strict interior of the SONC cone, thus only containing functions that are strictly positive; see \cref{corollary:NoZeros}.
In \cref{proposition:ExtremeRays} we classify the extreme rays of the DSONC cone, which turn out to behave analogously to the primal case.
Finally, we show that the DSONC cone is closed under affine transformation \cref{proposition:ClosureAffine}, but not under multiplication \cref{proposition:DSONCNotClosedUnderMultiplication}.
The latter is interesting, since (although the primal SONC cone is not closed under multiplication) the dual SONC cone is closed under multiplication.
We show that this property leads, however, to another kind of closure given by the standard inner product of the vector space of exponential sums over a finite support $A$; \cref{proposition:DSONCClosureInnerProduct}.

In \cref{section:DSONCandSOS} we investigate the relation of DSONC to SDSOS.
While it is immediately clear from known results \cite{Hartzer:Rohrig:deWolff:Yuruk:MMS,Iliman:deWolff:Circuits,Reznick:AGI} that, as in the general SONC case, the intersection of the DSONC and the SOS cone is governed by \struc{maximal mediated sets} (see \cref{definition:MaxMediatedSet} and \cref{corollary:DSONCandMMS} for details), it is not obvious how DSONCs relate to \struc{scaled diagonally dominant (SDSOS) polynomials} introduced by Ahmadi and Majumdar \cite{Ahmadi:Majumdar:DSOSandSDSOS}, which are located in the intersection of SOS and SONC.
In \cref{proposition:RelationDSONCandSDSOS} we show that for every fixed support set the SDSOS and the DSONC cone have non-empty intersection, but are not contained in each other.

In the final \cref{Section:StructuralViewpoints} we investigate DSONCs from a more abstract point of view.
A distinguished feature of circuit functions in the boundary of the SONC cone is that they have a single singular point (i.e., in particular a zero; in the case of circuit polynomials there is a single singular point on the positive orthant).
In \cite{Forsgaard:deWolff:BoundarySONCcone} Forsg\aa{}rd and the second author show that (for a fixed support) the coordinates of this point determine a circuit function entirely and that in consequence every such circuit function can be interpreted as an agiform in the sense of Reznick \cite{Reznick:AGI} joint with a group action.
Hence, as circuit functions in the boundary of the DSONC cone are strictly positive, an imminent question is whether there is a counterpart of the singular points and the group action.
We show that this is indeed the case. 
For every circuit function $f$ in the boundary of the DSONC cone there exists a distinguished, unique \struc{equilibrium point}, i.e., a point where all monomials attain the same value \cref{corollary:EquilibriaInnerTerm}, which relates the boundary of the DSONC cone to tropical geometry; see \cref{corollary:DSONCTropical}.
In \cref{theorem:EquilibraTorusAction} we conclude that, as in the primal case, assuming a fixed support, the coordinates of this point determine the circuit function, and every such circuit function can be interpreted as a function whose coefficient vector is the all-1-vector with a single negative sign in the last entry, joint with a group action.

\section*{Acknowledgments}

The authors were supported by the DFG grant WO 2206/1-1.

\section{Preliminaries}

\subsection{Notation}

Throughout this paper we refer to vectors by bold symbols. 
E.g., we write the vector $(x_1, \ldots, x_n) \in \R^n$ as $\struc{\Vector{x}}$.
We further use the short notation $\struc{\N^*}$ for $\N \setminus \{0\}$ (analogously for $\R^*$ etc.),
$\struc{\xb^{\alpb}}$ for $x_1^{\alp_1} \cdots x_n^{\alp_n}$, and $\struc{[n]}$ for $\{1,2,\ldots,n\}$ for all $n \in \N^*$.
If we restrict $\R$ to the positive orthant, then we write $\struc{\R_{>0}} = \{ x \in \R \ : \ x > 0 \}$.

For some given set $S \subseteq \R^n$ we denote the \struc{convex hull} of $S$ by $\struc{\conv(S)}$, the \struc{conic hull} of $S$ by $\struc{\cone(S)}$, and we refer to the set of \struc{vertices} of a convex set $C = \conv(S)$ by $\struc{\vertices{C}}$.
We further write $\struc{\relint{C}}$ for the \struc{relative interior} of some convex set $C$.
If $C$ is a given cone in $\R^n$, or a linear space, then we write $\struc{\check{C}}$ for the \struc{dual cone/space}.
We refer to the \struc{boundary} of a cone $C$ in $\R^n$ by $\struc{\partial C}$. 
For $S_1, \ldots, S_d \subseteq \R^n$ we denote the \struc{Minkowski sum} of $S_1, \ldots, S_d$ by $\struc{\sum_{i = 1}^d S_i}$.
For the (natural) logarithmic function $\ln(\cdot)$ and $x \in \R_{>0}$ we adhere to the common conventions $\ln(0) = - \infty$ and $\ln(\frac{0}{0}) = 0$, as well as $0\ln(\frac{0}{x}) = 0$ and $\ln(\frac{x}{0}) = \infty$.

\subsection{Signomial Optimization and the Signed SONC Cone}
\label{section:SignomialOptimizationSignedSONC}

Consider some finite set $A \subset \R^n$.
We define the \struc{set of exponential sums supported on $A$} as
\begin{align*}
	\struc{\R^A} \ = \ \spann_{\R} \left( \left\{ \expalpha \ : \ \alpb \in A \right\} \right).
\end{align*}
Thus, every \struc{exponential sum} (or \struc{signomial}) in $\R^A$ can be written as
\begin{align*}
	f(\xb) \ = \ \sum_{\alpb \in A} c_{\alpb} \expalpha 
\end{align*}
with \struc{coefficients} $c_{\alpb} \in \R$. 
We denote the \struc{coefficient vector of $f$} by $\struc{\coeff(f)} = \struc{\Vector{c}} = \left(c_{\alpb}\right)_{\alpb \in A}$.
We call $\struc{\supp(f)} = \set{\alpb \in A \subset \R^n \ : \ c_{\alpb} \ne 0}$ the \struc{support (set)} of $f$. 
The individual summands $c_{\alpb} \expalpha$ of $f$ are called \struc{exponential monomials}.
If the support of $f$ is fixed, then $f$ is uniquely defined by its coefficients.
Thus, for fixed $A$, there exists an isomorphism $\R^A \cong \R^{\# A}$,
and we can interpret the coefficient vector $\Vector{c}$ as an element in $\R^A$ as well.
This notation is referred to as \emph{``A-philosophy''} following Gelfand, Kapranov, and Zelevinsky, see \cite{Gelfand:Kapranov:Zelevinsky:Discriminants}, and also known as \emph{``fewnomial theory''}.

Whenever we fix the support set $A$ of a family of functions and additionally fix the orthant of $\R^A$ in which the coefficient vectors of this family are contained, we say that we assume a \struc{fixed sign distribution}.
Then, we are able to differentiate between the set $\struc{A^+}$ of exponents $\alpb$ in $A$ which have \emph{positive} corresponding coefficients $c_{\alpb}$, and the set $\struc{A^-}$ of exponents $\betab$ in $A$ which have \emph{negative} corresponding coefficients $c_{\betab}$, resulting in the decomposition 
\begin{align}
	A \ = \ A^+ \cup A^-
\label{eqn:SupportDecomposition}
\end{align}
of the support set $A$ on the particular orthant of $\R^A$.
When $A$ is the support of an exponential sum $f$, we also call $A^+$ the \struc{positive support} of $f$ and $A^-$ the \struc{negative support} of $f$.
Note that this implies $A^+ \cap A^- = \emptyset$.

Assuming knowledge of the signs of the coefficient vector of an exponential sum is common practice for studying exponential sums from the perspective of optimization, i.e.,  when one is interested in the \struc{global signomial optimization problem}
\begin{align}
	\inf_{\xb \in \R^n} f(\xb),
\label{prob:SignomialOpt}
\end{align}
see e.g. \cite{Dressler:Iliman:deWolff:FirstConstrained, Iliman:deWolff:FirstGP, Murray:Chandrasekaran:Wierman:NewtonPolytopes, Murray:Chandrasekaran:Wierman:SigOptREP}.
By rewriting \cref{prob:SignomialOpt} as
\begin{align}
	\sup \{\gamma \in \R \ : \ f(\xb) - \gamma \ge 0 \ \text{ for all } \xb \in \R^n \},
\label{prob:Nonnegativity}
\end{align}
we can turn any global signomial optimization problem into a problem of deciding containment in the \struc{nonnegativity cone}
\footnote{
Usually, in the context of polynomials we would denote by $\cP_{A}$ the cone of polynomials nonnegative on the entire $\R^n$, and write $\cP^+_{A}$ whenever we talk about polynomials being nonnegative on the positive orthant or about exponential sums.
Es we focus on exponential sums, we omit the $+$ indicating the positive orthant for brevity during the course of this paper for the nonnegativity cone and all subsets thereof.
}
\begin{align}
	\struc{\NonnegCone} \ = \ \left\{ f \in \R^A \ : \ f(\xb) \ge 0 \ \text{ for all } \xb \in \R^n \right\} \subset \R^A.
\label{definition:NonnegativityCone}
\end{align}
Note that the set $\NonnegCone$ is a full-dimensional closed convex cone in $\R^A$.
Since verifying containment in $\NonnegCone$ is an NP-hard problem, see e.g. \cite{Laurent:Survey}, we look for conditions on exponential sums that \emph{imply} nonnegativity.
We call such conditions \struc{certificates of nonnegativity}.
The certificate we are focusing on is the set of \struc{sums of nonnegative circuit functions (SONC)}. 
The name is motivated by the support set of these functions.

\begin{definition}[Circuits]
	A finite set $C \subset \R^n$ is called a \struc{circuit} if it is minimally affine dependent, see \cite{DeLoera:et:al:Triangulations}.
	If furthermore $\conv(C)$ is a simplex, then we call $C$ a \struc{simplicial circuit}.
	Here and in what follows we always assume that circuits are simplicial.
	We refer to the set of all (simplicial) circuits that can be constructed from some finite set $A \subset \R^n$ by $\struc{\cC(A)}$.

	A (simplicial) circuit $C = \vertices{C} \cup \{\betab\} \in \cC(A)$, where $\betab$ denotes the unique point in the interior of $\conv(C)$, is called \struc{minimal} in $A$ if 
	\begin{align*}
		(\conv(\vertices{C}) \setminus \vertices{C}) \cap A \ = \ \set{\betab}.
	\end{align*}
	The set of all minimal (simplicial) circuits that can be constructed from a finite set $A \subset \R^n$ is denoted by $\struc{\cCmin(A)} \subseteq \cC(A)$.
	Moreover, we also consider sets consisting of a single point to be circuits, i.e., $\{\alpb\} \in \cC(A)$ for all $\alpb \in A$.
\label{definition:Circuit}
\end{definition}

Minimal circuits are also referred to as \struc{reduced circuits}, see e.g. \cite{Murray:Naumann:Theobald:Sublinear}, or as \struc{$w$-thin}, see \cite{Reznick:AGI}.
Here, we follow the notation introduced in \cite{Forsgaard:deWolff:BoundarySONCcone}.

\begin{definition}[Circuit Functions]
	Let $A \subset \R^n$ be finite and let $f$ be an exponential sum supported on $A$.
	Then $f$ is called a \struc{circuit function} if it satisfies that 
	\begin{enumerate}
		\item $A$ is a simplicial circuit,
		\item if $\alpb \in \vertices{\conv(A)}$, then $c_{\alpb} > 0$, and
		\item if $\betab \in A \setminus \vertices{\conv(A)}$, then $\betab$ is in the strict interior of $\conv(A)$.
	\end{enumerate}
	If $A \subset \N^n$, then we speak about \struc{circuit polynomials} on $\R^n_{>0}$.
\label{definition:CircuitFunction}
\end{definition}

It is well-known that Condition (2) in \cref{definition:CircuitFunction} is a necessary condition for containment in $\NonnegCone$ for any function in $\R^A$, see \cite{Reznick:Extremal} and more recent works such as \cite{Feliu:Kaihnsa:Yueruek:deWolff:Multistationarity}.
Hence, we assume that
\begin{align}
	\vertices{\conv(A)} \subseteq A^+ \ \ne \ \emptyset
\label{assumption:Vertices}
\end{align}
is always fulfilled in the remainder of this paper.
Note that circuit functions $f$ supported on $A$ satisfy that, in addition to \cref{assumption:Vertices} being satisfied, their negative support $A^-$ can contain at most one point, namely $\betab$.
Since $\conv(A)$ is, by \cref{definition:CircuitFunction}, a simplex, there exist \emph{unique} \struc{barycentric coordinates} $\struc{\Vector{\lambda}} \in (0,1)^{\#A^+}$ such that
\begin{align*}
	\betab \ = \ \langle \Vector{\lambda}, \alpb \rangle \ = \ \sum_{\alpb \in A^+} \lambda_{\alpb} \alpb.
\end{align*}
We use this fact to decide nonnegativity of circuit functions.

\begin{theorem}[{\cite[Theorem 1.1]{Iliman:deWolff:Circuits}}]
	Let $f(\xb) = \sum_{\alpb \in A^+} c_{\alpb} \expalpha + c_{\betab} \expbeta$ be a circuit function with $\supp(f) = A = A^+ \cup A^-$, where $A^- = \{  \betab \} \ne \emptyset$.
	Let $\Vector{\lambda} \in (0,1)^{\#A^+}$ be the (unique) barycentric coordinates of $\betab$ with respect to $A^+$.
	Then $f$ is nonnegative if and only if 
	\begin{align*}
		\abs{c_{\betab}} \ \le \ \struc{\Theta_f} \ = \ \prod_{\alpb \in A^+} \left( \frac{c_{\alpb}}{\lambda_{\alpb}}\right)^{\lambda_{\alpb}}.
	\end{align*}
\label{theorem:NonnegativeCircuits}
\end{theorem}
The invariant $\Theta_f$ is called the \struc{circuit number} of $f$.
Note that if we had $\betab \in A^+$ and thus $A^- = \emptyset$ in the previous theorem, then $f$ was a sum of nonnegative exponential monomials, and as such also nonnegative.

\medskip

Every circuit function with $\betab \in A^-$ has a \emph{unique} extremal point, which is always a minimizer (in the polynomial case, this holds when restricting to the positive orthant).
This fact has been proven in \cite[Corollary 3.6]{Iliman:deWolff:Circuits} for the polynomial case, and in the signomial case it is a direct consequence of \cite[Theorem 3.4]{Craciun:Koeppl:Pantea:GlobalInjectivity}.

\begin{corollary}[{\cite{Iliman:deWolff:Circuits,Craciun:Koeppl:Pantea:GlobalInjectivity}}]
	Let $f$ be a circuit function supported on $A^+ \cup \set{\betab}$ as in \cref{eqn:SupportDecomposition}.
	Then $f$ has a unique minimizer.
\label{corollary:Minima}
\end{corollary}

We define the subset of $\NonnegCone$ containing all sums of nonnegative circuit (SONC) functions.

\begin{definition}[SONC Cone]
	Let $A \subset \R^n$ be a finite set. 
	The \struc{SONC cone} $\struc{\ASONC}$ is the cone of all exponential sums that can be written as sums of nonnegative circuit functions 
	with support in $A$. 
\label{definition:SONCCone}
\end{definition}

Whenever we restrict ourselves to SONC functions with fixed sign distributions, we consider the following intersection of $\ASONC$ with a specific orthant.

\begin{definition}[Signed SONC Cone]
	Let $A \subset \R^n$ be a finite set.
	Fix an arbitrary orthant of $\R^A$ yielding a decomposition $A = A^+ \cup A^-$ in the sense of \cref{eqn:SupportDecomposition}. 
	The \struc{signed SONC cone} $\struc{\signedSONC}$ is the cone of all exponential sums that can be written as sums of 
	nonnegative circuit functions 
	with support in $A^+ \cup A^-$. 
	We emphasize the special case $A^- = \{\betab\}$ as $\struc{\circuitSONC}$.
\label{definition:SignedSONCCone}
\end{definition}

It follows that the signed SONC cone can be written as the Minkowski sum
\begin{align}
	\signedSONC \ = \ \sum_{\betab \in A^-} \circuitSONC.
\label{eqn:PrimalMinkowski}
\end{align}

Describing the (signed) SONC cone does not require knowledge of explicit decompositions into circuit functions of the particular elements of the cone.
To see this, consider first the polytope
\begin{align}
	\struc{\Lambda(A^+, \betab)} \ = \ \left\{ \Vector{\lambda} \in \R_{\ge 0}^{A^+} \ : \ \betab = \sum_{\alpb \in A^+} \lambda_{\alpb} \alpb, \ \sum_{\alpb \in A^+} \lambda_{\alpb} = 1 \right\}.
\label{definition:LambdaPolytope}
\end{align}
Then the following theorem holds.

\begin{theorem}[{\cite[Theorem 2.7]{Katthaen:Naumann:Theobald:UnifiedFramework}}]  
	Let $A \subset \R^n$ be a fixed support set with a decomposition $A = A^+ \cup \{\betab\}$ as in \cref{eqn:SupportDecomposition}. 
	It holds that
	\begin{align*}
	\circuitSONC \ = \ \left\{\sum_{\alpb \in A^+} c_{\alpb} \expalpha + c_{\betab} \expbeta \ : \ 
	\begin{array}{c}
		\text{there exists } \Vector{\lambda} \in \Lambda(A^+, \betab) \text{ such that} \\
	\prod_{\alpb \in A^+: \lambda_{\alpb} >0} \left(\frac{c_{\alpb}}{\lambda_{\alpb}}\right)^{\lambda_{\alpb}} \ge - c_{\betab}
	\end{array}
	\right\}.
	\end{align*}
\label{theorem:SONCCone}	
\end{theorem}

This independence of an explicit decomposition into circuit functions motivates the following definition.

\begin{definition}[{\AGE}s]
	Let $A \subset \R^n$ be a fixed support set with a decomposition $A = A^+ \cup \{ \betab \}$ as in \cref{eqn:SupportDecomposition}.
	We call exponential sums $f$ of the form 
	\begin{align*}
		f(\xb) \ = \ \sum_{\alpb \in A^+} c_{\alpb} \expalpha + c_{\betab} \expbeta
	\end{align*}
	\struc{{\AGE}s}.
\end{definition}

Note that circuit functions are a special case of {\AGE}s.
The concept of studying {\AGE}s and \struc{sums of nonnegative {\AGE}s (SAGE)} under the viewpoint of signomial optimization was initiated by Chandrasekaran and Shah \cite{Chandrasekaran:Shah:SAGE-REP}, and was developed further in \cite{Murray:Chandrasekaran:Wierman:NewtonPolytopes,Murray:Chandrasekaran:Wierman:SigOptREP} by Chandrasekaran, Murray, and Wiermann; see also further developments by Katth\"an, Naumann, and Theobald in \cite{Katthaen:Naumann:Theobald:UnifiedFramework} using the name \struc{$\cS$-cone}. 

\medskip

In fact, the SONC and SAGE cones describe the same object.
For the special case of \struc{agiforms}, this has already been proven by Reznick in \cite{Reznick:AGI} in 1989.
Agiforms are AGE functions whose coefficient vector is contained in $\Lambda(A^+, \betab)$ up to a positive scalar multiple.
Wang implicitly showed the equality of the SONC and SAGE cones in \cite{Wang:nonnegative}.
The first explicit statement of the ``SONC = SAGE" result was given independently shortly after in \cite{Murray:Chandrasekaran:Wierman:NewtonPolytopes}.
The same statement can also be derived from results in \cite{Katthaen:Naumann:Theobald:UnifiedFramework} and \cite{Forsgaard:deWolff:BoundarySONCcone}.

Moreover, every function in the SONC/SAGE cone $\ASONC$ can be decomposed into circuit functions or {\AGE}s, respectively, whose supports do not contain points outside the initial support set $A$. 
In the language of circuit polynomials, this was proven by Wang in \cite{Wang:nonnegative}.
The corresponding theorem for the setting of {\AGE}s was shown in \cite{Murray:Chandrasekaran:Wierman:NewtonPolytopes}, and implies the following statement as a special case.

\begin{theorem}[\cite{Wang:nonnegative,Murray:Chandrasekaran:Wierman:NewtonPolytopes}]
	Let $f \in \ASONC$. 
	Then $f$ has a representation as a sum of nonnegative circuit functions $f_k$ whose supports satisfy $\supp(f_k) \subseteq \supp(f)$.
\label{theorem:SONCSupportDecomposition}
\end{theorem}

\subsection{The dual SONC cone}

In this subsection we study the dual of the SONC cone $\signedSONC$.
This cone was previously studied in, e.g., \cite{Katthaen:Naumann:Theobald:UnifiedFramework,Dressler:Heuer:Naumann:deWolff:DualSONCLP}.
The following theorem gives three different representations of the dual SONC cone using the \struc{natural duality pairing}
\begin{align*}
	\struc{\Vector{v}(f)} \ = \ \sum_{\alpb \in A^+} v_{\alpb} c_{\alpb} +\sum_{\betab \in A^-} v_{\betab} c_{\betab} \ \in \ \R,	
\end{align*}
where $\Vector{v}(\cdot) \in \check \R^A$.

\begin{theorem}[The dual SONC cone; \cite{Dressler:Heuer:Naumann:deWolff:DualSONCLP}]
	Let $A= A^+\cup A^-$ as in \cref{eqn:SupportDecomposition}.
	Then the following sets are equal.

	\renewcommand{\arraystretch}{1.3}
	\begin{align*}
		\begin{aligned}[c]
			(1) \qquad & \struc{\dualsignedSONC}  & \ = 
			& \; \left\{ \; \Vector{v} \in \check \R^A \ : \ \Vector{v}(f) \geq 0 \text{ for all } f \in \signedSONC \; \right\}, \\ 
			(2) \qquad & & 
			& \; \left\{ \; \Vector{v} \in \check \R^A \ : \ 
			\begin{array}{c}
				v_{\alpb} \ge 0 \text{ for all } \; \alpb \in A^+ \text{, and } \\ \ln (|v_{\betab}|) \ \le \ \sum\limits_{\alpb \in A^+} \lambda_{\alpb} \ln(v_{\alpb}) \\ \text{ for all } \betab \in A^-, 
				\text{ and all } \; \Vector{\lambda} \in \Lambda(A^+, {\betab}) \;
			\end{array}
			\right\} \\
			(3)\qquad & & 
			& \; \left\{\Vector{v} \in \check \R^A \ : \ 
			\begin{array}{c}
				v_{\alpb} \ge 0 \text{ for all } \; \alpb \in A^+ \text{, and} \\
				\text{ for all } \; \betab \in A^- 
				\text{ there exists } \Vector{\tau} \in \R^n \text{ such that } \\ \ln \left( \frac{| v_{\betab}|}{v_{\alpb}} \right) \ \le \ (\alpb - \betab)\T \Vector{\tau} \; \text{ for all } \; {\alpb} \in A^+\; 
			\end{array}
			\right\} 
		\end{aligned}
	\end{align*}
\label{theorem:DualCone}
\end{theorem}

Another way of representing the dual SONC cone is given by the following corollary.

\begin{corollary}[{\cite[Corollary 3.4]{Dressler:Heuer:Naumann:deWolff:DualSONCLP}}]
	Let $A = A^+ \cup A^-$ be a finite support set such that $A^-$ is nonempty.
	Then it holds that
	\begin{align*}
		\signedSONC \ = \ \sum\limits_{\betab \in A^-} \circuitSONC \ , \qquad \text{ and } \qquad \dualsignedSONC \ = \ \bigcap\limits_{\betab \in A^-} \dualcircuitSONC.
	\end{align*}
	\label{corollary:DualCone}
\end{corollary}

To verify containment in the dual SONC cone $\dualsignedSONC$, it is sufficient to investigate the minimal circuits $\cCmin(A)$.

\begin{theorem}[{\cite[Theorem 3.5]{Katthaen:Naumann:Theobald:UnifiedFramework}}]
	Consider a finite set $A = A^+ \cup A^- \subset \R^n$ and a linear functional $\Vector{v} \in \check\R^A$.
	Then $\Vector{v} \in \dualsignedSONC$ if and only if $v_{\alpb} \ge 0 \ $ for all $\alpb \in A^+$, and for every circuit $C^+ \cup \set{\betab} \in \cC(A)$ the inequality
	\begin{align}
		\ln(v_{\betab}) \ \le \ \sum_{\alpb \in C^+} \lambda_{\alpb} \ln(v_{\alpb})
	\label{eqn:DualContainment}
	\end{align}
	holds for the unique $\Vector{\lambda} \in \Lambda(C^+, \betab)$.
	This in turn is equivalent to the condition that \cref{eqn:DualContainment} holds for every \emph{minimal} circuit $C^+ \cup \set{\betab} \in \cCmin(A)$.
\label{theorem:ReducedDualContainment}
\end{theorem}

This reduction to the case of minimal circuits directly parallels analogous results for the primal SONC cone and other related objects, see e.g. \cite[Theorem 4]{Murray:Chandrasekaran:Wierman:NewtonPolytopes}, \cite[Theorem 2.4]{Forsgaard:deWolff:BoundarySONCcone}, or \cite[Corollary 8.11]{Reznick:AGI}.

\section{The DSONC Cone}
\label{section:DSONCIntroduction}

In this section we introduce the key object of this article.
We can use the dual SONC cone described in the previous section to define a new subcone of $\NonnegCone$, which we call the \struc{DSONC cone}.

Recall first, that we use the natural duality pairing $\Vector{v}(f) = \sum_{\alpb \in A} v_{\alpb} c_{\alpb}$, where $f = \sum_{\alpb \in A} c_{\alpb} \expalpha \in \ASONC$, to represent elements $\Vector{v}(\cdot)$ in the dual SONC cone $\dualASONC$. 
We can thus interpret $\Vector{v}$ as a vector $(v_{\alpb})_{\alpb \in A} \in \R^A$. 
For $A = A^+ \cup \set{\betab}$, we thus use the identification
\begin{align}
	g(\xb) \ = \ \sum_{\alpb \in A^+} v_{\alpb} \expalpha + v_{\betab} \expbeta
\label{eqn:Identification}
\end{align}
to associate {\AGE}s $g \in \R^A$ with elements $\Vector{v} \in \dualcircuitSONC$.
Using this, we can identify the cone $\dualcircuitSONC$ with the cone of all {\AGE}s of the form~\cref{eqn:Identification} that have a coefficient vector $\Vector{v} \in \dualcircuitSONC$. 
We refer to this cone by $\struc{\dualcircuitSONCfunctions}$. 
Together with \cref{corollary:DualCone}, this motivates the following definition.

\begin{definition}[The DSONC Cone]
	For every finite support set $A \subset \R^n$ we define the \struc{DSONC cone} as the set 
	\begin{align*}
		\struc{\ADSONC} \ = \ \cone \set{ f \text{ is \AGE } \ : \ \coeff(f) \in \dualASONC}
	\end{align*}
	containing sums of {\AGE}s whose coefficients are contained in the dual SONC cone.
	If we consider a fixed sign distribution $A = A^+ \cup A^-$, then we write $\struc{\signedDSONC}$ for the \struc{signed DSONC cone}.
\label{definition:DSONC}
\end{definition}

Note that for support sets of the special form $A = A^+ \cup \set{\betab}$ the signed DSONC cone equals precisely the dual of the SONC cone $\dualcircuitSONC$, up to identification of its elements with exponential sums. 
Hence, we write $\circuitDSONC$ for both objects.
In consequence, we can infer from \cref{definition:DSONC} that
\begin{align}
	\signedDSONC \ = \ \sum_{\betab \in A^-} \dualcircuitSONCfunctions.
	\label{eqn:DualMinkowski}
\end{align}
For arbitrary support sets, however, this direct identification of $\signedDSONC$ with $\dualsignedSONC$ no longer holds.

\begin{example}
	To illustrate the different behavior of elements in the dual SONC and DSONC cone, consider $A^+ = \left\{ \begin{pmatrix}0\\0\end{pmatrix}, \begin{pmatrix}4\\0\end{pmatrix}, \begin{pmatrix}0\\4\end{pmatrix}\right\}$, and $A^- = \set{\betab_k \ : \ k \in [m]}$ for some $m \in \N_{> 0}$, where
    \begin{align*}
        \betab_k \ = \ \frac{1}{k+1} \begin{pmatrix}4\\0\end{pmatrix} + \frac{1}{k+1} \begin{pmatrix}0\\4\end{pmatrix} + \left(1 - \frac{2}{k+1}\right) \begin{pmatrix}0\\0\end{pmatrix}.
    \end{align*}
	Associated to this fixed support set we investigate the vector $\Vector{c} = (1, 1, 1, -1, \ldots, -1) \in \R^A$ by first interpreting it as an element in the dual SONC cone $\dualsignedSONC$, and second as the coefficient vector of an exponential sum.

	To see that $\Vector{c}$ is contained in $\dualsignedSONC$ we need to check the conditions in \cref{theorem:DualCone} for each $\betab_k$. 
	Since every $\betab_k$ has the unique barycentric coordinates $\Vector{\lambda} = \left(\frac{1}{k+1}, \frac{1}{k+1}, 1 - \frac{2}{k+1} \right)$,
    \begin{align*}
        \ln(\abs{v_{\betab_k}}) \ = \ \ln(1) \ = \ 0 \ \le \ \frac{2}{k+1} \ln(1) + \left(1 - \frac{2}{k+1}\right) \ln(1) \ = \ 0
    \end{align*}
	holds for all $k \in [m]$, and $\Vector{c} \in \dualsignedSONC$ follows.

	If we take $\Vector{c}$ to be a coefficient vector, then the resulting exponential sum is
    \begin{align*}
        f \ = \ e^{4x_1} + e^{4x_2} + 1 - \sum_{k = 1}^m e^{\langle \xb, \betab_k \rangle}. 
    \end{align*}
	This function has $m$ negative terms, which means that for $m > 3$, $f$ is clearly not nonnegative and thus not contained in the primal SONC cone.
	Since the DSONC cone is always contained in the SONC cone, see \cref{corollary:DSONCinPrimal}, this already implies that $f$ cannot be contained in the DSONC cone. 

	The reason why the direct correspondence between $\signedDSONC$ and $\dualsignedSONC$ breaks for $\# A^- > 1$ is that containment in the dual SONC cone checks the containment criteria for each inner point $\betab$ independently. 
	This means that neither the number nor the combined weight of the negative terms are significant for containment in the dual SONC cone, but both of these attributes are highly relevant for studying nonnegative exponential sums, and DSONC functions as a special case thereof.
	\label{example:DualSONCNotDSONC}
\end{example}

It was shown in \cite[Proposition 3.6]{Dressler:Heuer:Naumann:deWolff:DualSONCLP} that every {\AGE} $f$ with coefficient vector in $\dualASONC$ is contained in the primal SONC cone and thus nonnegative. 
This implies the following corollary.
\begin{corollary}
	Let $A \subset \R^n$ be finite. 
	Then the DSONC cone $\ADSONC$ is contained in the primal SONC cone $\ASONC$. 
\label{corollary:DSONCinPrimal}
\end{corollary}

Thus, containment in the DSONC cone is a certificate of nonnegativity. 

\begin{proof}
	Consider an arbitrary sign distribution $A = A^+ \cup A^- \subset \R^n$.
	If $A^-$ is empty, then $\signedDSONC$ only contains sums of nonnegative exponential monomials and the claim follows.
	Assume now that $A^- \ne \emptyset$.
	By \cref{eqn:DualMinkowski} we have that $\signedDSONC = \sum_{\betab \in A^-} \dualcircuitSONCfunctions$.
	According to \cite[Proposition 3.6]{Dressler:Heuer:Naumann:deWolff:DualSONCLP} it holds that 
	$\dualcircuitSONCfunctions \subset \circuitSONC$, and thus
	\begin{align*}
		\signedSONC \ = \ \sum\limits_{\betab \in A^-} \circuitSONC \ \supset \ \sum_{\betab \in A^-} \dualcircuitSONCfunctions \ = \ \signedDSONC \ .
	\end{align*}
	Since our choice for the sign distribution was arbitrary, the claim follows.
\end{proof}

A natural question at this point is to ask whether we can find an isomorphism between $\dualsignedSONC$ and $\signedDSONC$.
This is not the case, as for general support sets with $\# A^- > 1$ we have that for all $\bar{\betab} \in A^-$ it holds that
\begin{align}
	\signedDSONC \ = \ \sum_{\betab \in A^-} \circuitDSONC \ \supset \ \DSONC_{A^+, \bar{\betab}} \ \cong \ \dualSONC_{A^+, \bar{\betab}} \ \supset \ \bigcap_{\betab \in A^-} \dualcircuitSONC \ = \ \dualsignedSONC \ .
\label{eqn:InclusionChain}
\end{align}

\medskip

We close this section by emphasizing that the DSONC cone can equivalently be defined as the cone of all sums of circuit functions which are supported on a (minimal) circuit and have a coefficient vector that is contained in the dual SONC cone. 
This means that, as in the primal case, the DSONC cone has equivalent representations in terms of circuit and {\AGE}s.
The following corollary directly parallels the definition of the SONC cone given in \cref{definition:SignedSONCCone}, whose circuit-less representation is stated in \cref{theorem:SONCCone}.
Thus, \cref{corollary:DSONCisDSAGE} can be seen as the DSONC version of the ``SONC = SAGE'' result.

\begin{corollary}[``DSONC = DSAGE'']
	For any finite support set $A \subset \R^n$ it holds that
	\begin{align*}
		\ADSONC \ = \ \cone \set{ f \text{ is minimal circuit function } \ : \ \coeff(f) \in \dualASONC}. 
	\end{align*}
\label{corollary:DSONCisDSAGE}
\end{corollary}

\begin{proof}
	This is an immediate consequence of \cref{theorem:ReducedDualContainment}.
\end{proof}

\section{Properties of the DSONC Cone}
\label{section:PropertiesDSONCCone}

As the DSONC cone is an object that, to the best of our knowledge, has not yet been studied in the existing literature, we establish its fundamental properties.
As we discover in this section, the DSONC cone demonstrates several structural similarities to the primal SONC cone.
We start with the following proposition.

\begin{proposition}
	The DSONC cone is a proper (closed, convex, pointed) full-dimensional cone.
	\label{proposition:DSONCconeProper}
\end{proposition}

\begin{proof}
	It is clear by definition that the DSONC cone is proper.
	We now need to show that it is full-dimensional.
	For this, we first fix an arbitrary signed support set $A = A^+ \cup \set{\betab}$ and identify elements in the DSONC cone $\circuitDSONC$ and dual SONC cone $\dualcircuitSONC$ with (coefficient) vectors $\Vector{c} \in \R^A$.
	If the closure of any convex cone is pointed, then its dual is full-dimensional, see e.g. \cite[Section 2.13.2.1, pp. 130]{Dattorro:ConvexOpt}.
	Since the primal SONC cone $\circuitSONC$ equals the nonnegativity cone for the case $A = A^+ \cup \set{\betab}$, it is a proper, full-dimensional cone.
	Thus, its dual is full-dimensional.
	The dual SONC cone and the DSONC cone for support sets $A^+ \cup \set{\betab}$ describe the same object when using the above identification, so $\circuitDSONC$ itself is also full-dimensional in $\R^{A^+ \cup \set{\betab}}$. 
	For general support sets $A^+ \cup A^-$ with some fixed arbitrary sign distribution, $\dualsignedSONC$ is full-dimensional by analogous arguments, and by the inclusions given in \cref{eqn:InclusionChain} it follows that $\signedDSONC$ is full-dimensional as well. 
	Since the sign distribution was chosen arbitrarily, the claim follows.
\end{proof}

Now, we discuss important basic properties and the strong connection that exists between the DSONC cone and the primal SONC cone.
Moreover, we explore a key difference between these two objects; namely, that elements in the DSONC cone cannot have real zeros.
We proceed by giving a description of the extreme rays of the DSONC cone.
Next, we investigate how affine transformation of variables affects DSONC functions and, last, show that the DSONC cone is not closed under multiplication.

\subsection{Basic Properties of the DSONC Cone and its Relation to the SONC Cone}

First, we observe that, as in the primal case, see \cref{theorem:NonnegativeCircuits}, we can describe the DSONC cone in terms of an invariant which we call the \struc{dual circuit number}.

\begin{definition}
	Let $f(\xb) = \sum_{\alpb \in A^+} c_{\alpb} \expalpha \ + \ c_{\betab} \expbeta$ be a circuit function with $\supp(f) = A^ + \cup \set{\betab}$ and let $\Vector{\lambda}$ be the unique vector in $\Lambda(A^+, \betab)$.
	Then we associate to $f$ the \struc{dual circuit number}
	\begin{align*}
		\struc{\check{\Theta}_f} \ = \ \prod_{\alpb \in A^+} c_{\alpha}^{\lambda_{\alpb}}.
	\end{align*}
\label{definition:DualCircuitNumber}
\end{definition}

We now make the following observation. 

\begin{remark}
	For the special case where the support set $A = A^+ \cup \set{\betab} \subset \R^n$ is a circuit it follows from representation (2) in \cref{theorem:DualCone} that
	\begin{align*}
		\dualcircuitSONC \ = \ 
		\set{\Vector{v} \in \check\R^A \ : \ v_{\alpb} > 0 \text{ for all } \alpb \in A^+, v_{\betab} \le 0, \text{ and } \ \abs{v_{\betab}} \le \prod_{\alpb \in A^+} v_{\alpb}^{\lambda_{\alpb}}},
	\end{align*}
	where $\Vector{\lambda}$ denotes the unique vector in $\Lambda(A^+, \betab)$.
\end{remark}

\begin{corollary}
	Let $f(\xb) = \sum_{\alpb \in A^+} c_{\alpb} \expalpha \ + \ c_{\betab} \expbeta$ be a circuit function with $\supp(f) = A = A^+ \cup \set{\betab}$.
	Then $f$ is an element in the DSONC cone corresponding to $A$ if and only if 
	\begin{align}
		- c_{\betab} \ \le \ \check\Theta_f.
	\label{eqn:DualCircuitNumber}
	\end{align}
\label{corollary:DualCircuitNumber}
\end{corollary}

\begin{proof}
	Since $f$ is a circuit function, $\conv(A)$ is a simplex.
	Thus, the barycentric coordinates $\Vector{\lambda} \in \Lambda(A^+, \betab)$ are unique.
	By definition of the DSONC cone, $f$ is contained in $\circuitDSONC$ if and only if its coefficient vector is contained in $\dualcircuitSONC$.
	Representation (2) of the dual SONC cone in \cref{theorem:DualCone} yields that the coefficient vector $\Vector{c}$ corresponds to an element in $\dualcircuitSONC$ if and only if 
	\begin{enumerate}
		\item $c_{\alpb} \ \ge \ 0 \ $ for all $\alpb \in A^+$, and
		\item $\ln(|c_{\betab}|) \ \le \ \sum_{\alpb \in A^+} \lambda_{\alpb} \ln(c_{\alpb}) \ $ if $c_{\betab} \ < \ 0$.
	\end{enumerate}
	The claim follows after applying $\exp(\cdot)$ to the inequality in (2).
\end{proof}

\begin{example}
	Let $\struc{m_{c,d}(x, y)} = x^4y^2 + x^2y^4 - cx^2y^2 + d$ denote the (family of) \struc{(generalized) Motzkin polynomial(s)} with inner coefficient $c$ and constant term $d$.
	Then we build the corresponding exponential sum by considering $m_{c,d}(e^x, e^y)$.
	Note that $m_{c,d}(e^x, e^y)$ is
	\begin{enumerate}
		\item a circuit function whenever $d > 0$,
		\item unbounded when $d < 0$, or $d = 0$ and $c > 0$,
		\item a trivially nonnegative sum of nonnegative exponential monomials when $c < 0$ and $d \ge 0$.
	\end{enumerate}
	Let us thus assume that $c, d > 0$. 
	The dual circuit number condition of $m_{c,d}$ is given by
	\begin{align}
		c \ \le \ 1^{\frac{1}{3}} \cdot 1^{\frac{1}{3}} \cdot d^{\frac{1}{3}},
	\label{eqn:MotzkinCircuitNumber}
	\end{align}
	where we have used that the inner point $(2,2)$ is located in the barycenter of 
	\begin{align*}
		\conv\left(\set{(4,2), (2, 4), (0, 0)}\right),
	\end{align*} 
	i.e., the barycentric coordinates are $\Vector{\lambda} = (1/3, 1/3, 1/3)$.
	It is clear that any circuit function with coefficient vector of the form $(1, \ldots, 1, -1)$, assuming that the coefficient $-1$ corresponds to the inner term, trivially satisfies \cref{eqn:DualCircuitNumber} and is thus contained in the DSONC cone.

	Let us investigate the case $c = 3$.
	For the additional choice $d = 1$ we get the classical Motzkin polynomial.
	It is well-known that $m_{3,1}$ is a nonnegative polynomial which is contained in the boundary of the primal SONC cone.
	It is, however, not contained in the DSONC cone, because \cref{eqn:MotzkinCircuitNumber} becomes $3 \le 1$, which is false.
	We can make \cref{eqn:MotzkinCircuitNumber} feasible by varying $c$ and adjusting $d$ accordingly (or vice versa).
	The two feasible parameter choices $(c, d) = (3, 27)$ and $(c, d) = (1, 1)$ are displayed in~\cref{figure:Motzkins}.
	\label{example:Motzkin}
\end{example}

\begin{center}
	\begin{figure}[hbt]~
		\includegraphics[draft=false, width=.75\textwidth]{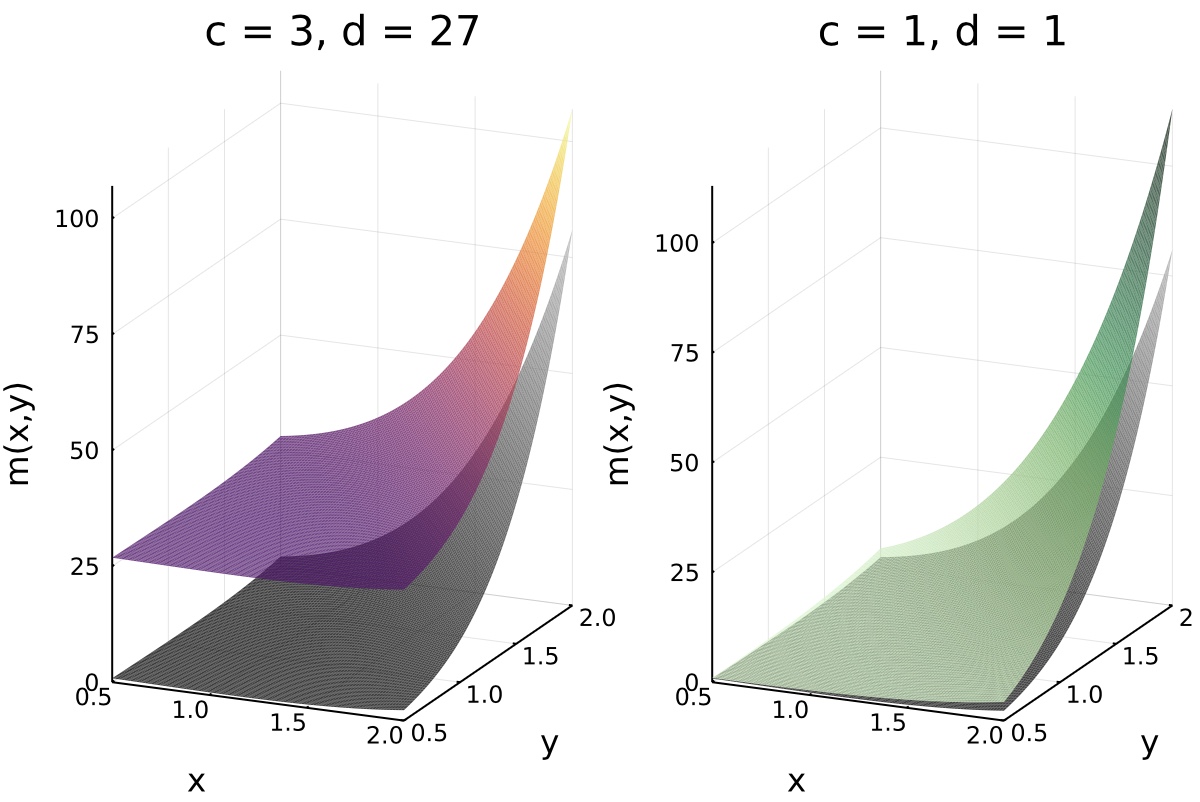}
		\caption{$m_{c,d}(x, y)$ for different parameters with $m_{3,1}(x,y)$ plotted in gray for reference.}
		\label{figure:Motzkins}
	\end{figure}
\end{center}

Observe that for every circuit function $f$ it holds that $\check{\Theta}_f < \Theta_f$, which once more illustrates the strict containment of the DSONC cone in the primal SONC cone. 
The same inference can also be made from \cref{corollary:DualBoundApprox}.
Indeed, we can write $\ADSONC$ for some finite support set $A = A^+ \cup A^-$ as
\begin{align*}
	\ADSONC \ = \ \cone \set{ f \text{ is circuit function } \ : \ -c_{\betab} \le \check\Theta_f }.
\end{align*}
We point out that from this representation one can directly observe that the cone $\circuitDSONC$ is closely related to an object called the \struc{power cone}, see e.g. \cite{Chares:PowerCone}. 
For $\Vector{\delta} \in \R^A_{> 0}$ with $\sum_{\alpb \in A} \delta_{\alpb} = 1$ and some finite $A \subset \R^n$, this object is defined as
\begin{align}
	\struc{\mathrm{P}_{\Vector{\delta}}} \ = \ \set{(c, d) \in \R^A_{> 0} \times \R\ : \ |d| \le \prod_{\alpb \in A} c_{\alpb}^{\delta_{\alpb}}}. \label{equation:PowerCone}
\end{align}
Specifically, if we restrict $\circuitDSONC$ to circuit functions which additionally satisfy that the coefficient corresponding to $\betab$ is negative and identify the cone $\circuitDSONC$ with the corresponding set of coefficient vectors, then we get the subset 
\begin{align*}
	\set{(c, d) \in \R^A_{> 0} \times \R_{< 0}\ : \ |d| \le \prod_{\alpb \in A} c_{\alpb}^{\lambda{\alpb}}}
\end{align*} 
of the power cone $\mathrm{P}_{\Vector{\lambda}}$, where $\Vector{\lambda}$ denotes the unique vector in $\Lambda(A^+, \betab)$.
Additionally, the dual of the power cone is given by the set
\begin{align*}
	\struc{\check{\mathrm{P}}_{\Vector{\delta}}} \ = \ \set{(u, v) \in \R^A_{> 0} \times \R\ : \ |v| \le \prod_{\alpb \in A} \left(\frac{u_{\alpb}}{\delta_{\alpb}}\right)^{\delta_{\alpb}}}.
\end{align*}
Thus, we can find the same relation as described above between the primal cone $\circuitSONC$ and $\check{\mathrm{P}}_{\Vector{\lambda}}$.

\medskip

Following \cref{corollary:DualCircuitNumber}, there is an easy way to construct a circuit function in the DSONC cone from any given nonnegative circuit function.
\begin{corollary}
	Let $f(\xb) = \sum_{\alpb \in A^+} c_{\alpb} \expalpha + c_{\betab} \expbeta$ be a circuit function with $\supp(f) = A^+ \cup \set{\betab}$.
	Let $\Vector{\lambda}$ be the unique vector in $\Lambda(A^+, \betab)$.
	Then $f \in \circuitSONC$ if and only if
	\begin{align*}
		\check{f}(\xb) \ = \ \sum_{\alpb \in A^+} \frac{c_{\alpb}}{\lambda_{\alpb}} \expalpha + c_{\betab} \expbeta \ \in \ \circuitDSONC.
	\end{align*}
\label{corollary:PrimalDualConnection}
\end{corollary}

\begin{proof}
	If $f \in \circuitSONC$, then we have by \cref{theorem:NonnegativeCircuits} that 
	\begin{align*}
		-c_{\betab} \ \le \ \Theta_f \ = \ \prod_{\alpb \in A^+} \left(\frac{c_{\alpha}}{\lambda_{\alpb}}\right)^{\lambda_{\alpb}} \ = \ \check{\Theta}_{\check{f}}.
	\end{align*}
	By \cref{corollary:DualCircuitNumber}, it follows that $\check{f} \in \circuitDSONC$.
	The reverse direction follows analogously.
\end{proof}

We further conclude that for every primal SONC certificate providing a lower bound to some polynomial / exponential sum, there exists a corresponding DSONC bound.

\begin{corollary}
	Let $A = A^+ \cup \set{\betab} \subset \R^n$ be a circuit satisfying $\betab = \Vector{0} \in \relint{A^+}$ and fix some coefficient vector $\Vector{c} \in \R^A$ such that $c_{\alpb} \ge 0$ for all $\alpb \in A^+$.
	If we define
	\begin{align*}
		\gamma_{\text{SONC}} & \ = \ \sup\set{\gamma \in \R \ : \ \sum_{\alpb \in A^+} c_{\alpb} \expalpha - \gamma \in \circuitSONC}, \text{ and} \\
		\gamma_{\text{DSONC}} & \ = \ \sup\set{\gamma \in \R \ : \ \sum_{\alpb \in A^+} c_{\alpb} \expalpha - \gamma \in \circuitDSONC},
	\end{align*}
	then it holds that 
	\begin{align}
		\gamma_{\text{SONC}} \ = \ \left( \prod_{\alpb \in A^+} \lambda_{\alpb}^{-\lambda_{\alpb}} \right) \cdot \gamma_{\text{DSONC}},
	\label{eqn:DualBoundFromSONC}
	\end{align}
	where $\Vector{\lambda}$ denotes the unique vector in $\Lambda(A^+, \Vector{0})$.
\label{corollary:DualBoundApprox}
\end{corollary}

Note that the assumption $\betab = \Vector{0}$ in the previous proposition is not restrictive, as any circuit function satisfies this condition after an affine transformation $\alpb \mapsto \alpb - \betab$ for all exponents $\alpb \in A$, and such a transformation leaves the dual circuit number invariant.

\begin{proof}[Proof of \cref{corollary:DualBoundApprox}]
	By \cref{theorem:NonnegativeCircuits} and \cref{corollary:DualCircuitNumber} it holds that 
	\begin{align*}
		\gamma_{\text{SONC}} & \ = \ \prod_{\alpb \in A^+} \left(\frac{c_{\alpb}}{\lambda_{\alpb}}\right)^{\lambda_{\alpb}} 
		\ = \ \left( \prod_{\alpb \in A^+} \left(\frac{1}{\lambda_{\alpb}}\right)^{\lambda_{\alpb}} \right) \cdot \left( \prod_{\alpb \in A^+ } c_{\alpb}^{\lambda_{\alpb}} \right) \ = \ \left( \prod_{\alpb \in A^+} \lambda_{\alpb}^{-\lambda_{\alpb}} \right) \cdot \gamma_{\text{DSONC}}. 
	\end{align*}
\end{proof}

\begin{example}
	Consider again the exponential sum
	\begin{align*}
		m_{3,1}(e^x, e^y) \ = \ e^{4x + 2y} + e^{2x + 4y} - 3 e^{2x + 2y} + 1
	\end{align*}
	coming from the Motzkin polynomial.
	To enforce the condition of \cref{corollary:DualBoundApprox} that the inner point must be the origin, we shift the exponents by $(-2, -2)$, i.e., we investigate
	\begin{align*}
		\bar m = m_{3,1}(e^x, e^y) \cdot e^{-2x -2y} \ = \ e^{2x} + e^{2y} - 3 + e^{-2x -2y}.
	\end{align*}
	It is well-known that $m_{3,1}$ is a circuit function with $m_{3,1} \ge 0 = -(1 + \gamma_{\text{SONC}})$, and that this minimum is attained at $(e^x,e^y) = (1,1)$.
	The same holds true for $\bar m$.
	The barycentric coordinates of the inner term $\betab =\Vector{0}$ with respect to $A^+ = \set{(2, 0)^\top, (0,2)^\top, (-2,-2)^\top}$ are $\Vector{\lambda} = (1/3, 1/3, 1/3)$.
	With this, \cref{eqn:DualBoundFromSONC} becomes
	\begin{align*}
		-1 \ = \ \frac{1}{27} \cdot \gamma_{\text{DSONC}}.
	\end{align*}
	The DSONC bound found in this way corresponds precisely to the dual SONC bound of the Motzkin polynomial as shown in \cite{Dressler:Heuer:Naumann:deWolff:DualSONCLP}.
\label{example:MotzkinBound}
\end{example}

\begin{remark}
	Recall that we can compute the DSONC bound of an exponential sum $f$ via \emph{linear} programming, see \cite{Dressler:Heuer:Naumann:deWolff:DualSONCLP}, which makes the connection in \cref{corollary:DualBoundApprox} particularly interesting for future computations. 
	When a circuit decomposition corresponding to the optimal SONC bound to $f$ is known, we can solve LPs to find DSONC bounds to those circuit functions and then effortlessly improve this bound using the relation given in \cref{corollary:DualBoundApprox} and obtain the optimal primal SONC bound. 

	If we do not compute an explicit decomposition into circuits, which is the case for the dual SONC computations in \cite{Dressler:Heuer:Naumann:deWolff:DualSONCLP}, then we can still make use of \cref{corollary:DualBoundApprox} in the following way.
	Computing lower bounds via the DSONC cone requires that we decompose $f$ into AGE functions, and then apply representation (3) from \cref{theorem:DualCone}. 
	This means that the barycentric coordinates $\Vector{\lambda}$ for each of these AGE functions are, in general, no longer unique. 
	However, for every involved AGE function containing the constant term, we can, by \cref{theorem:SONCCone}, choose \emph{any} corresponding $\Vector{\lambda}$ and still get an improved but not necessarily optimal SONC bound via \cref{eqn:DualBoundFromSONC}.
\label{remark:DualBoundApprox}
\end{remark}

\begin{remark}
	We conclude this subsection by emphasizing that, while it is true that we can use linear programming to \emph{optimize} over the DSONC cone, the DSONC cone itself is \emph{not} polyhedral!
	We obtain a linear program because when optimizing a given exponential sum, we consider all coefficients except the constant term to be fixed. 
	This turns the nonlinear description of the boundary of the DSONC cone into a linear condition.
	\label{remark:LinearProgramming}
\end{remark}

\subsection{Zeros of Elements in the DSONC Cone}

It is a direct consequence of \cref{corollary:DualCircuitNumber} and the subsequent observation that functions in the DSONC cone cannot have a real zero.
Thus, the DSONC cone does not intersect the boundary of the cone of nonnegative exponential sums.

To show this, we briefly recall that in \cite[Proposition 3.4 and Corollary 3.9]{Iliman:deWolff:Circuits} Iliman and the second author showed that if $A = A^+ \cup \set{\betab} \subset \N^n$ is a circuit, then the nonnegative circuit function
\begin{align*}
	f_{\betab}(\xb) \ = \ \sum_{\alpb \in A^+} c_{\alpb} \expalpha + c_{\betab} \expbeta \ \in \ \circuitSONC
\end{align*} 
has a real zero if and only if $c_{\betab} = - \Theta_{\betab}$.
This immediately generalizes to the case $A \subset \R^n$, and, using \cref{theorem:SONCCone}, also to the case of {\AGE}s.
In other words, $f_{\betab} \in \circuitSONC$ has a real zero if and only if it is contained in $\partial \circuitSONC$.
The same does not hold for the DSONC cone.

\begin{corollary}
	Let $A \subset \R^n$ be some fixed finite support set and let $f$ be an exponential sum in $\ADSONC$.
	Then
	\begin{align*}
		\ADSONC \cap \ \partial\ASONC = \emptyset.
	\end{align*}
	In particular, this shows that 
	\begin{align*}
	\text{if } f(\xb) \in \ADSONC \ \text{, then } \ f(\xb) > 0 \text{ for all } \xb \in \R^n,
\end{align*}
	i.e., $f$ has no real zeros.
\label{corollary:NoZeros}
\end{corollary}

\begin{proof}[Proof of \cref{corollary:NoZeros}]
	Let $A = A^+ \cup A^-$ with an arbitrary support decomposition as in \cref{eqn:SupportDecomposition}. 
	By \cite[Proposition 4.4]{Forsgaard:deWolff:BoundarySONCcone}, we have that if $f \in \partial \signedSONC$, then $f$ has a real zero.
	We thus assume for sake of contradiction that $f$ is an exponential sum in $\signedDSONC$ which has a real zero.
	By definition of $\signedDSONC$, $f$ can be written as a sum of nonnegative circuit functions
	\begin{align*}
		f(\xb) \ = \ \sum_{\betab \in A^-} f_{\betab}(\xb),
	\end{align*}
	where $f_{\betab} \in \circuitDSONC$.
	Thus, $f$ can only have a real zero $\Vector{z}$, if $f_{\betab}(\Vector{z}) = 0$ for all $\betab \in A^-$.
	This is the case if and only if the coefficients $c_{\betab}$ corresponding to the inner terms are equal to $-\Theta_{f_{\betab}}$ by \cite[Proposition 3.4 and Corollary 3.9]{Iliman:deWolff:Circuits}.
	The claim now follows immediately from \cref{corollary:DualCircuitNumber}, which says that since all the $f_{\betab}$ are elements in $\circuitDSONC$, it must hold that
	\begin{align*}
		|c_{\betab}| \ \le \ \check{\Theta}_{f_{\betab}} \ < \ \Theta_{f_{\betab}}.
	\end{align*}
	Since the support decomposition is arbitrary, the claim follows for general support sets as well.
\end{proof}

We close this subsection by pointing out that the existing results on minimizers of circuit functions, see \cref{corollary:Minima}, especially also holds for circuit functions in the DSONC cone.

\subsection{Extreme Rays of the Dual SONC Cone}

The following proposition gives a description of the extreme rays of the DSONC cone.
Similarly to the case illustrated in \cref{corollary:DualBoundApprox}, the extreme rays of the DSONC cone can be inferred from those of the primal SONC cone by replacing each involved circuit number with its corresponding dual circuit number.

\begin{proposition}
	Let $A \subset \R^n$ be a finite support set.
	Consider the set $\cCmin(A)$ of all minimal circuits contained in $A$.
	The extreme rays of $\ADSONC$ are given by the set
	\begin{align*}
		E_1 & \ \cup \ E_2,
	\end{align*}
	where
	\begin{align*}
		\struc{E_1} & \ = \ \bigcup_{C^+ \cup \set{\betab} \in \cCmin(A)} \set{\sum_{\alpb \in C^+} c_{\alpb} \expalpha - \left(\prod_{\alpb \in C^+}c_{\alpb}^{\lambda_{\alpb}}\right) \expbeta \ : \ \Vector{c} \in \R^{C^+}_{>0}}, \; \text{ and }\\
		\struc{ E_2} & \ = \ \bigcup_{\Vector{\gamma} \in A} \set{d e^{\langle \xb, \Vector{\gamma} \rangle} \ : \ d \in \R_{>0}}.
	\end{align*}
\label{proposition:ExtremeRays}
\end{proposition}

The proof of \cref{proposition:ExtremeRays} directly follows \cite[Proof of Proposition 4.4]{Katthaen:Naumann:Theobald:UnifiedFramework}.
We include a detailed proof in \cref{section:Appendix} for completeness.

\begin{example}[The univariate Case]
	Consider a univariate (minimal) circuit $A = \set{a_0, a_1, a_2}$, where $a_0 < a_1 < a_2$.
	Then the inner point is $a_1$ and any circuit function supported on $A$ is of the form 
	\begin{align*}
		f(x) \ = \ c_0e^{a_0x} + c_1e^{a_1x} + c_2e^{a_2x}.
	\end{align*}
	According to \cref{proposition:ExtremeRays}, whenever $f$ is contained in the boundary of the DSONC cone, it holds that
	\begin{align*}
		|c_{\betab}| \ = \ c_0^{\lambda_0} \cdot c_2^{\lambda_2},	
	\end{align*}
	where the barycentric coordinates $(\lambda_0, \lambda_2)$ are given by 
	\begin{align*}
		\lambda_0 \ = \ \frac{a_2-a_1}{a_2-a_0}, \qquad
		\lambda_2 \ = \ \frac{a_1-a_0}{a_2-a_0}.
	\end{align*}
	Using this, we can describe the boundary of the DSONC cone for fixed support sets of type $A$, see e.g. \cref{figure:UnivariateBoundary}.
\end{example}

\begin{center}
	\begin{figure}[hbt]
			\includegraphics[draft=false, width=.75\textwidth]{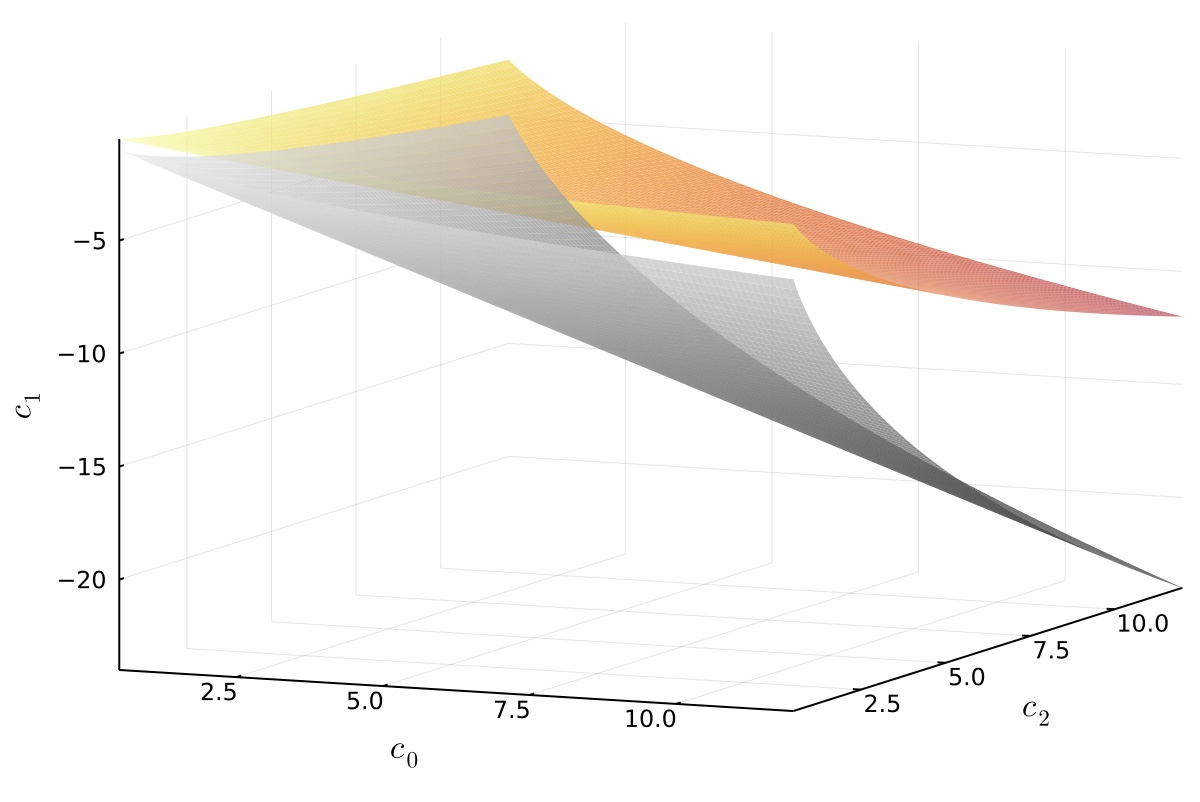}
			\caption{Section of the boundary of the univariate DSONC cone for fixed support set $A = \set{0,2,4}$, including boundary of the corresponding primal SONC cone in gray for comparison.}
		\label{figure:UnivariateBoundary}
	\end{figure}
\end{center}

The following closing observation about extreme rays of the DSONC cone is crucial for the more abstract viewpoints on the DSONC cone, which we provide in \cref{Section:StructuralViewpoints}.

\begin{corollary}
	Let $f = \sum_{\alpb \in A^+} c_{\alpb} \expalpha + c_{\betab} \expbeta \in E_1$ be a minimal circuit function with support $A^+ \cup \set{\betab} \in \cCmin(A)$ for some finite support set $A \subset \R^n$. 
	Then $f$ is an extreme ray in $\partial \ADSONC$ if and only if there exists some $\Vector{\tau} \in \R^n$ such that
	\begin{align*}
		|c_{\betab}|e^{\langle \Vector{\tau}, \betab \rangle} \ = \ c_{\alpb} e^{\langle \Vector{\tau}, \alpb \rangle} \qquad \text{ for all } \; \alpb \in A^+.
	\end{align*}
\label{corollary:Boundary}
\end{corollary}

\begin{proof}
	The statement is an immediate consequence of \cref{theorem:DualCone}.
	On the boundary of the DSONC cone the inequalities in representation (3) are satisfied with equality, which yields that
	\begin{align*}
		\ln\left(\frac{|c_{\betab}|}{c_{\alpb}}\right) \ = \ (\alpb - \betab)\T \Vector{\tau}
	\end{align*}
	holds for some $\Vector{\tau} \in \R^n$ and all $\alpb \in A^+$.
	Applying $\exp(\cdot)$ to both sides of the equality 
	yields the desired result.
\end{proof}

\subsection{Closures}

We now show that the DSONC cone is closed under affine transformation of variables. 
The same holds true for the primal SONC case when considering exponential sums, as implicitly noted in \cite{Iliman:deWolff:Circuits}.

\begin{proposition} 
	Let $f(\xb)$ be a circuit function in the DSONC cone.
	Consider an affine transformation $T \ : \ \R^n \rightarrow \R^n$ given by $T(\xb) = \Vector{Mx} + \Vector{a}$ for some transformation matrix $\Vector{M} \in \mathrm{GL}(n, \R)$ and a translation vector $\Vector{a} \in \R^n$.
	Then $f(T(\xb))$ is also contained in the DSONC cone.
\label{proposition:ClosureAffine}
\end{proposition}

The proof of this proposition is a straightforward computation, which we provide for completeness in \cref{section:Appendix}.
Note that this proposition does not carry over to the case of polynomials on the entire $\R^n$, as in this case the situation has to be monitored orthant by orthant.
For the primal SONC cone this has been shown in \cite[Corollary 3.2]{Dressler:Kurpisz:deWolff:Hypercube}.

Another well-known feature of the primal SONC cone of polynomials on $\R^n$ is that it is not closed under multiplication, see \cite{Dressler:Iliman:deWolff:Positivstellensatz, Dressler:Kurpisz:deWolff:Hypercube}. 
The same is true for the DSONC cone.
This is not a trivial observation, as there exist products of circuits polynomials, which are circuit polynomials again.
Only, this property does not hold in general.

\begin{proposition}
	The DSONC cone is not closed under multiplication.
	\label{proposition:DSONCNotClosedUnderMultiplication}
\end{proposition}

\begin{proof}
	We provide a simple counterexample.
	Consider $f_1(x_1, x_2) = 1 + e^{2x_1} - e^{x_1}$ and $f_2(x_1, x_2) = 1 + e^{2x_2} - e^{x_2}$.
	Then $f_1, f_2$ are functions on the boundary of the DSONC cone, that is, the coefficients of their inner terms are equal to their respective dual circuit numbers.
	By reproducing the proof of \cite[Lemma 3.1]{Dressler:Kurpisz:deWolff:Hypercube} for $f_1 \cdot f_2$, we see that the product $f_1 \cdot f_2$ is not contained in the DSONC cone.
\end{proof}

However, by the basic definition of the dual cone, $\dualsignedSONC$ is closed under multiplication.

\begin{proposition}
	The cone $\dualsignedSONC$ is closed under multiplication.
\label{proposition:DualSONCClosureMultiplication}
\end{proposition}

\begin{proof}
	Let $\Vector{v_1}, \Vector{v_2} \in \dualsignedSONC$.
	Then, by \cref{theorem:DualCone}, Part (1), it holds that $\Vector{v_1}(f) \ge 0$ and $\Vector{v_2}(f) \ge 0$ for all $f \in \signedSONC$.
	Thus, $(\Vector{v_1} \cdot \Vector{v_2})(f) = \Vector{v_1}(f) \cdot \Vector{v_2}(f) \ge 0$ for all $f \in \signedSONC$.
	By definition of the dual cone, it follows that $\Vector{v_1} \cdot \Vector{v_2} \in \dualsignedSONC$.
\end{proof}

While this property does not translate to general functions in $\signedDSONC$, we can still use it to construct elements in the DSONC cone.

\begin{proposition}
	Let $f, g \in \circuitDSONC$ with $\coeff(f) = \Vector{c}$ and $\coeff(g) = \Vector{d}$ such that $c_{\betab}, d_{\betab} < 0$.
	Then 
	\begin{align*}
		h(\xb) \ = \ \sum_{\alpb \in A^+} c_{\alpb} d_{\alpb} e^{\langle \xb , \alpb \rangle} - c_{\betab} d_{\betab} e^{\langle \xb , \betab \rangle} \ \in \ \circuitDSONC. 	 
	\end{align*}
	\label{proposition:DSONCClosureInnerProduct}
\end{proposition}

Note that this operation is given by the standard inner product of the real vector space of exponential sums with support $A$ with respect to the monomial basis, which for $p = \sum_{\alpb \in A} c_{\alpb} \expalpha$ and $q = \sum_{\alpb \in A} d_{\alpb} \expalpha$, where $\alpb \in \R^n$, is defined as $\langle p, q \rangle \ = \ \sum_{\alpb \in A} c_{\alpb} d_{\alpb}$.

\begin{proof}
	Since $f,g \in \circuitDSONC$, it holds for every $\Vector{\lambda} \in \Lambda(A^+, \betab)$ that
	\begin{align*}
		\begin{aligned}[c]
			\ln(\abs{c_{\betab}}) & \le \sum_{\alpb \in A^+} \lambda_{\alpb} \ln(c_{\alpb}), & & \text{ and } & &
		\ln(\abs{d_{\betab}}) & \le \sum_{\alpb \in A^+} \lambda_{\alpb} \ln(d_{\alpb}).
		\end{aligned}
	\end{align*}
	It follows that
	\begin{align*}
		\ln(\abs{- c_{\betab} d_{\betab}}) \ = \ \ln(\abs{c_{\betab}}) + \ln(\abs{d_{\betab}}) \ \le \ \sum_{\alpb \in A^+} \lambda_{\alpb} (\ln(c_{\alpb}) + \ln(d_{\alpb})) \ = \ \sum_{\alpb \in A^+} \lambda_{\alpb} \ln(c_{\alpb} d_{\alpb}).
	\end{align*}
	Thus $h \in \circuitDSONC$.
\end{proof}

\section{When are DSONC Polynomials Sums of Squares?}
\label{section:DSONCandSOS}

In this section we exclusively consider the polynomial case, i.e., the case where $A \subset \N^n$.
Recall that in this case $\R^A$ can be identified with the space of polynomials in $\xb \in \R_{>0}$ with support $A$ via the bijective mapping $\xb = (x_1, \ldots, x_n) \mapsto (e^{x_1}, \ldots, e^{x_2})$.
Together with the fact that circuit polynomials are only nonnegative on $\R^n$ if they are nonnegative on $\R^n_{>0}$, see \cite[Section 3.1]{Iliman:deWolff:Circuits}, the results of the previous sections (except \cref{proposition:ClosureAffine}) can immediately be adapted to the case of polynomials on the entire $\R^n$.

The connection and intersection between the SONC approach to polynomial nonnegativity and the sum of squares (SOS) approach has been studied in, e.g., \cite{Reznick:AGI,Kurpisz:deWolff:Hierarchies,Iliman:deWolff:Circuits, Hartzer:Rohrig:deWolff:Yuruk:MMS}.
It has been proven by Iliman and the second author in \cite[Theorem 5.2]{Iliman:deWolff:Circuits} that a nonnegative circuit polynomial $p$ supported on $A$ can be written as a sum of squares if and only if the inner point $\betab$ is contained in the \struc{maximal mediated set (MMS)} of $\conv(A)$, or if $p$ is a sum of monomial squares, extending a proof by Reznick of the same statement for agiforms in \cite{Reznick:AGI}.
Note that containment in the MMS is an entirely combinatorial condition which only depends on the support of the polynomial and is independent of the choice of coefficients of the involved polynomials.

\begin{definition}
	Let $\Delta \subset \R^n$ be a set of points such that $\conv(\Delta)$ is a simplex and let $M \subseteq \conv(\Delta) \cap \Z^n$ be a subset of lattice points.
	Then $M$ is called \struc{$\Delta$-mediated} if
	\begin{enumerate}
		\item $\Delta \subset M$, and
		\item for any given $\Vector{\delta} \in M \setminus \Delta$ there exist points $\Vector{\delta}_1, \Vector{\delta}_2 \in (2\Z)^n \cap M$ such that $\Vector{\delta}_1 \ne \Vector{\delta}_2$ and $\Vector{\delta} = \frac{1}{2}\Vector{\delta}_1 + \frac{1}{2}\Vector{\delta}_2$, i.e., $\Vector{\delta}$ is the midpoint of two distinct even lattice point in $M$.
	\end{enumerate}
	The largest set satisfying (1) and (2) is called the \struc{maximal mediated set} of $\Delta$ and is denoted by $\struc{\Delta^*}$.
	\label{definition:MaxMediatedSet}
\end{definition}

For a comprehensive introduction we refer the interested reader to \cite{Reznick:AGI,Hartzer:Rohrig:deWolff:Yuruk:MMS,Yuruk:thesis}.
Using results about MMS, we obtain a characterization of all polynomials in the DSONC cone which can be written as sums of squares as a special case.

\begin{corollary}
	Let $p(\xb)$ be a DSONC polynomial supported on a finite set $\Delta = A^+ \cup A^-$, such that the convex hull of $A^+ \subset (2\N)^n$ is a simplex and $A^- \subseteq \mathrm{relint}(\conv(A^+) \cap \N^n)$.
	Then $f$ is a sum of squares if and only if either every $\betab \in A^-$ satisfies $\betab \in \Delta^*$ or $p$ is a sum of monomial squares.
	In particular, there exist DSONCs which are SOS but not sums of monomial squares.
	\label{corollary:DSONCandMMS}
\end{corollary}

\begin{proof}
	By \cite[Theorem 3.9]{Hartzer:Rohrig:deWolff:Yuruk:MMS}, the claim holds for any SONC polynomial $p$, and thus it especially also holds for DSONC polynomials.
\end{proof}

\subsection{The Relation of SDSOS and DSONC}

One particularly interesting subset of the SOS cone is the cone of \struc{scaled diagonally dominant (SDSOS) polynomials}.
The SDSOS cone was first introduced my Ahmadi and Majumdar in \cite{Ahmadi:Majumdar:DSOSandSDSOS} and can be interpreted as the cone of those polynomials which can be written as sums of \emph{binomial} squares.
I.e., if $p$ is SDSOS, then we can write $p$ as 
\begin{align}
	p(\xb) \ = \ \sum_{j = 1}^{k} p_j(\xb)
\label{eqn:SDSOSRepresentation}
\end{align} 
with binomial squares $p_j(\xb) = (a_j p_j(\xb) + b_j q_j(\xb))^2$,
where $a_j, b_j \in \R$, and $p_j$ and $q_j$ are monomials, see \cite{Ahmadi:Majumdar:DSOSandSDSOS, Ahmadi:Hall:BasisPursuit}.
Since both the SDSOS and the DSONC cone are contained in the primal SONC cone (see, e.g., \cite{Kurpisz:deWolff:Hierarchies}, and \cref{corollary:DSONCinPrimal}), it is natural to ask whether SDSOS is a subset of DSONC (the converse is not true, since DSONC is not even contained in the SOS cone).
In this section we show that this is not the case.

\begin{proposition}
	Let $p = (a p(\xb) + b q(\xb))^2$ be a binomial square, where $a, b \in \R$, and $p$ and $q$ are monomials.
	Then $p$ is contained in the DSONC cone if and only if it is a sum of monomial squares.
\label{proposition:BinSquareNotDualSONC}
\end{proposition}

\begin{proof}
	Since $p = (a p(\xb) + b q(\xb))^2$ is a binomial square with $a, b \in \R$, and monomials $p$ and $q$, we can find $\alpb, \betab \in \N^n$ such that
	\begin{align}
		p(\xb) \ = \ \left( a \xb^{\alpb} + b \xb^{\alpb} \right)^2 \ = \ a^2 \xb^{2\alpb} + b^2 \xb^{2 \betab} + 2ab \xb^{\alpb + \betab}.
	\label{eqn:BinomialSquare}
	\end{align}
	In the cases where $a,b \in \R_{\ge 0}$ or $a,b \in \R_{\le 0}$, it is clear from \cref{eqn:BinomialSquare} that $p$ is a sum of monomial squares.
	Consider now the case where the signs of $a,b \ne 0$ differ, i.e. the case where the coefficient $2ab$ is negative. 
	Denote the support of $p$ by $A = \{2 \alpb, 2 \betab \} \cup \{ \alpb + \betab\}$.
	Note that $A$ is a circuit.
	If $p$ was contained in $\ADSONC$, then we had for the unique $\Vector{\lambda} \in \Lambda(A^+, \alpb + \betab)$ that
	\begin{align*}
		\ln(\abs{2ab}) \ \le \ \lambda_{\alpb} \ln(a^2) + \lambda_{\betab} \ln(b^2)
	\end{align*}
	by \cref{theorem:DualCone} (2).
	However, $\Vector{\lambda} = (1/2, 1/2)$ and it holds that
	\begin{align*}
		\ln(\abs{2a b}) & \ > \ \frac{1}{2} \ln(a^2) + \frac{1}{2} \ln(b^2) \ = \ \ln(\abs{a b}). 
	\end{align*}
	Thus, $p \notin \circuitDSONC$.
\end{proof}

This essentially shows that the DSONC cone and the SDSOS cone have different building blocks. 
Indeed, we can show that these two cones do not contain each other but have nonempty intersection.

\begin{proposition}
	Let $A \subset \N^n$ be a general fixed finite support set.
	For the cone $\mathrm{SDSOS}_{A}$ of SDSOS functions supported on $A$ and the cone $\ADSONC$ of DSONC polynomials it holds that
	\begin{align*}
		\mathrm{SDSOS}_{A} \ \not\subset \ \ADSONC, \qquad \ADSONC \ \not\subset \ \mathrm{SDSOS}_{A}, \qquad \text{ and } \qquad \ADSONC \ \cap \ \mathrm{SDSOS}_{A} \ \ne \ \emptyset.
	\end{align*}
\label{proposition:RelationDSONCandSDSOS}
\end{proposition}

\begin{proof}
The first claim follows directly from \cref{proposition:BinSquareNotDualSONC}.
For the second claim consider $m_{3, 27}(x,y) = x^{4}y^{2} + x^{2}y^{4} - 3x^{2}y^{2} + 27$, which is in the DSONC cone, see \cref{example:Motzkin}.
The support of $m_{3, 27}$ yields the most prominent and easiest example of a maximal mediated set that does not contain interior points.
Therefore, no polynomial in the family of generalized Motzkin polynomials discussed in \cref{example:Motzkin} is SOS; for further discussion see e.g. \cite{Hartzer:Rohrig:deWolff:Yuruk:MMS,Kurpisz:deWolff:Hierarchies,Reznick:AGI,Yuruk:thesis}.
This particularly implies that $m_{3, 27} \notin \mathrm{SDSOS}_{A}$.

For the last statement, we give an example of a polynomial which is contained in $\ADSONC$ and $\mathrm{SDSOS}_{A}$.
Let $p(x) = 5x^{2} + 5x^{6} - 8x^{4}$.
Then $\supp(p) = A$ is a circuit and to show that $p \in \ADSONC$, we can verify the inequality given by \cref{theorem:DualCone} (2).
This holds since
\begin{align*}
	\ln \left(\frac{|-8|}{5} \right) & \ \le \ \frac{1}{2} \ln(5) + \frac{1}{2} \ln(5) \ = \ \ln(5).
\end{align*}
Since furthermore all coefficients corresponding to vertices of the convex hull of $A$ are positive, $p$ is an element in the DSONC cone.
Conversely, we can write $p$ as
\begin{align*}
	p & \ = \ 5x^{2} + 5x^{6} - 8x^{4} \ = \ \left(2x - x^{3}\right)^2 + \left(x - 2x^{3}\right)^2,
\end{align*}
so $p$ is also SDSOS.
\end{proof}

\section{An Abstract Viewpoint on the DSONC Cone}
\label{Section:StructuralViewpoints}

We return once more to the case of general finite support sets $A \subset \R^n$, i.e., to the case of exponential sums.
From \cite{Forsgaard:deWolff:BoundarySONCcone} we know that singular circuit functions and the SONC cone can be viewed at from the following (more) abstract point of view. 
Let $A$ be of the form 
\begin{align}
	A \ = \ \set{\alpb_0, \ldots, \alpb_d}
\label{eqn:NumberedSupport}
\end{align}
for some $d \in \N$, and consider the \struc{exponential toric morphism} $\varphi_A \colon \R^n \rightarrow \R^{d + 1}$
whose coordinates are exponential monomials:
\begin{align*}
	\struc{\varphi_A(\x)} \ = \ \left( e^{\langle\x, \alpb_0\rangle}, \dots, e^{\langle\x, \alpb_d\rangle} \right).
\end{align*}

Then the space $\R^A$ is a given by
\begin{align}
	\R^A \ = \ \set{f(\x) =  \langle\varphi_A(\x), \Vector{c}\rangle \ : \ \Vector{c} \in \R^{d+1}},
\label{eqn:f}
\end{align}
i.e., every exponential sum in $\R^A$ is a composition of $\varphi_A$ and linear forms acting on $\R^{d+1}$.
In particular, there is a group action $G: \R^n\times \R^A \rightarrow \R^A$, given by $(\wb, f) \mapsto f(\x - \wb)$.
Let us interpret this action. 
Assume that $f$ is a singular nonnegative circuit function supported on the circuit $C \in \cC(A)$.
Then, as in \cref{section:SignomialOptimizationSignedSONC}, $\R^A$ is isomorphic to $\R^{d+1}$ by identifying every exponential sum with its coefficient vector $\Vector{c}$. 
Thus, on the one hand, $f$ (and its singular point $\Vector{s} \in \R^n$) are determined by the pair $(C,\Vector{c})$.
On the other hand, the vector of coefficients $\Vector{c}$ is also determined by the pair $(C, \Vector{s})$, i.e., if a circuit is given, then there exists \emph{exactly one} singular nonnegative circuit function with singular point $\Vector{s}$.
Simplicial agiforms (in the sense of Reznick) play a special role as they are exactly those singular nonnegative circuit functions with a singular point at $(0, 0, \ldots, 0)$.
In this sense, one can think about singular nonnegative circuit functions as ``\emph{agiforms plus a group action}''.

\medskip

For the polynomial case, it was already described in \cite{Iliman:deWolff:Circuits} that a nonnegative circuit polynomial $f$ has a zero in $(\R^*)^n$ if and only if $f$ belongs to the $A$-discriminant in $\C^A \supset \R^A$, which is an algebraic hypersurface on $\C^A$ given by all polynomials that admit a singular point.

Following \cite{Forsgaard:deWolff:BoundarySONCcone}, this can be interpreted and adapted for exponential sums as follows.
Let $f(\x) =  \langle\varphi_A(\x), \Vector{c}\rangle$ be supported on $A$ as in \cref{eqn:NumberedSupport} and investigate
\begin{align}
	\left[\begin{array}{ccc}  1 & \cdots & 1 \\ \alpb_0 & \cdots & \alpb_d \end{array}\right] \cdot \left(\Vector{c} * \varphi_A(\x)\right)\T ,
\label{eqn:SetA}
\end{align}
where $\struc{*}$ denotes the component-wise product.
Then the first row of the resulting vector is $f(\x)$ itself and the $(i+1)$-st row equals $x_i \cdot \frac{\partial f}{\partial x_i}(\x)$.
Therefore, $f$ admits a singular point if and only if the vector $\left(\Vector{c} * \varphi_A(\x)\right)$ belongs to the kernel of $A$ for a suitable choice of $\x \in \R^{n}$.
In order to make the property of having a singular point independent of choosing such a point $\x \in \R^{n}$ (or more general $\x \in \C^n$ in the context of arbitrary $A$-discriminants) one quotients out the group action $G$, which exactly corresponds this choice of a singular point.
The resulting hypersurface is called the \struc{reduced $A$-discriminant}; for further information see \cite{Gelfand:Kapranov:Zelevinsky:Discriminants}.

\medskip

Let now $f(\x) =  \langle\varphi_A(\x), \Vector{c}\rangle$ be contained in the DSONC cone.
As $f(\x)$ is strictly positive, there is no singular point, and $f$ is not contained in the $A$-discriminant and does not belong to the kernel of $A$ in the upper sense.
Hence, in this section we discuss
\begin{enumerate}
	\item What is the counterpart of singular points for circuit functions $f$ in the boundary of the DSONC cone (i.e., satisfying $\check \Theta_f = c_{\betab}$)?
	\item What role plays the group action $G$? Is there a counterpart of  the ``\emph{agiforms plus a group action}'' characterization of SONCs for the DSONC cone?
\end{enumerate}
For remainder of this section let $f(\xb) = \sum_{j = 0}^n c_j \xb^{\alpb_j} + c_{\betab} \xb^{\betab}$ be a circuit function with $A^+ = \{\alpb_0, \ldots, \alpb_n\}$.

\begin{definition}
	Let $f(\x)$ be as before.
	We define the \struc{point of equilibrium $\eq{f} \in \R^n$} as the unique point satisfying
	\begin{align*}
		c_0 e^{\langle \eq{f}, \alpb_0 - \betab \rangle} \ = \ \cdots \ = \ c_n e^{\langle \eq{f}, \alpb_n - \betab \rangle}.
	\end{align*}
	\label{definition:Equilibrium}
\end{definition}

Note that the uniqueness of this point comes from the fact that $A^+$ forms a simplex.
We make the following observation.

\begin{corollary}
	The circuit function $f$ belongs to the boundary of the dual SONC cone, i.e., the inequality \cref{eqn:DualCircuitNumber} of the dual circuit number is satisfied with equality if and only if 
	\begin{align}
		|c_{\betab}| \ = \ c_0 e^{\langle \eq{f}, \alpb_0 - \betab \rangle} \ = \ \cdots \ = \ c_n e^{\langle \eq{f}, \alpb_n - \betab \rangle}. 
	\label{Equation:Equilibrium}
	\end{align}
\label{corollary:EquilibriaInnerTerm}
\end{corollary}

\begin{proof}
	This is an immediate consequence of \cref{corollary:Boundary}.
\end{proof}

It follows that the previous condition can be expressed in terms of \struc{tropical geometry}; see e.g. \cite{Maclagan:Sturmfels} for a general introduction to the topic.

\begin{corollary}
	The circuit function $f$ belongs to the DSONC cone if and only if the tropical hypersurface $\cT$ given by
	\begin{align}
		\bigoplus_{\alpb \in A^+} \log|c_{\alpb}| \odot \expalpha \oplus \log|c_{\betab}| \odot \expbeta
	\end{align}
	has genus zero, i.e., its complement contains no bounded connected component.
	\label{corollary:DSONCTropical}
\end{corollary}

The proof is straightforward and also follows essentially from \cite[Lemma 3.4 (a)]{Theobald:deWolff:Genus1}.
We give the key steps without discussing tropical geometry in detail here.

\begin{proof}
	It is sufficient to consider the case that $f$ belongs to the boundary of the DSONC cone.
	Then, as $\conv(A)$ is a simplex, the $n+1$ outer hyperplanes of $\cT$ intersect in a unique point $\xb^*$, which has to satisfy
	\begin{align*}
		c_0 e^{\langle \xb^*, \alpb_0 - \betab \rangle} \ = \ \cdots \ = \ c_n e^{\langle \xb^*, \alpb_n - \betab \rangle},
	\end{align*}
	i.e., $\xb^* = \eq{f}$ by \cref{definition:Equilibrium}.
	By \cref{Equation:Equilibrium} it follows that the $f$ belongs to the boundary of the DSONC cone if and only if the inner term does not dominate at $\eq{f}$, i.e., the complement of $\cT$ does not have a bounded component.
\end{proof}

This corollary can be seen as a geometrical manifestation of the somewhat surprising fact that membership of an individual circuit function in the DSONC cone is a linear condition as we discussed in \cref{remark:LinearProgramming}.

\begin{example}~
	\begin{enumerate}
		\item Consider $m_{c, d}(e^x, e^y) =  e^{4x+2y} + e^{2x + 4y} - c e^{2x + 2y} + d$, where $m_{c, d}(x, y)$ is the (generalized) Motzkin polynomial as in \cref{example:Motzkin}.
		For the case $(c, d) = (3, 1)$, this function attains its minimal value in the singular point $\Vector{s} = (0, 0)$ and is not contained in the DSONC cone.
		We have already seen that $m_{3, 27}(e^x, e^y)$ is a circuit function on the boundary of the DSONC cone, and we can compute its point of equilibrium as the solution $(\tau_1^\ast, \tau_2^\ast) \in \R^2$ to the system of equations
		\begin{align}
			e^{2\tau_1} \ = \ e^{2\tau_2} \ = \ 27 e^{-2\tau_1 - 2\tau_2},
		\label{eqn:EquilibriumExample1}
		\end{align}
		which yields $\eq{m_{3, 27}(e^x, e^y)} = (\tau_1^\ast, \tau_2^\ast) = (\ln(\sqrt{3}), \ln(\sqrt{3}))$.
		Note that since we are on the boundary of the DSONC cone, evaluating \cref{eqn:EquilibriumExample1} at this point yields the absolute value of the inner coefficient ($c = 3$); see \cref{corollary:Boundary}.
		If we switch to the case $(c,d) = (1, 1)$, then \cref{eqn:EquilibriumExample1} becomes
		\begin{align*}
			e^{2\tau_1} \ = \ e^{2\tau_2} \ = \ e^{-2\tau_1 - 2\tau_2},
		\end{align*}
		so $\eq{m_{1,1}(e^x, e^y)} = (0, 0)$.
		\item It is particularly easy to calculate the equilibria for circuit functions whose coefficient vector is (a positive scalar multiple of) $(1, \ldots, 1, -1) \in \R^n$, where we assume that the last coordinate corresponds to the inner term.
		As discussed in \cref{example:Motzkin}, such circuit functions are trivially contained in the boundary of the DSONC cone.
		Since in this case the equilibrium has to satisfy 
		\begin{align*}
			e^{\langle \eq{f}, \alpb_0 - \betab \rangle} \ = \ \cdots \ = \ e^{\langle \eq{f}, \alpb_n - \betab \rangle} \ = \ |c_{\betab}| \ = \ 1,
		\end{align*}
		we have $\eq{f} = (0, \ldots, 0)$.
	\end{enumerate}
	\label{example:Equilibria}
\end{example}

\begin{proposition}
	Let $f(\xb) = \sum_{j = 0}^n c_j e^{\langle \xb, \alpb_j - \betab \rangle} + c_{\betab}$ be a circuit function whose support set forms an $n$-simplex.
	Then the minimizer $\xb_{\ast}$ of $f$ satisfies $\xb_{\ast} = \eq{f}$ if and only if $|c_{\betab}| = \frac{1}{n+1}$, and $\supp(f)$ is a \struc{barycentric circuit}, i.e., all barycentric coordinates coincide.
\label{proposition:MinimizerEquilibria}
\end{proposition}

We emphasize that this proposition also holds in the special case, where $\xb_{\ast} = \Vector{s}$ is a singular point of $f$, i.e. when $f$ is contained in the boundary of the primal SONC cone. 

\begin{proof}[Proof of \cref{proposition:MinimizerEquilibria}]
	Note first that $f$ has a unique extremal point, which is always a minimizer, see \cref{corollary:Minima}.
	Thus, there exists exactly one point $\xb_{\ast}$ which satisfies
	\begin{align*}
		x_i \frac{\partial f}{\partial x_i} \ = \ \sum_{j = 0}^n c_j ( \left(\alpha_j\right)_i - \beta_i) e^{\langle \xb, \alpb_j - \betab \rangle} \ = \ 0 \qquad \text{ for all } i \in [n].
	\end{align*}
	On the other hand, we know that there exists a (unique) vector $\Vector{\lambda} \in (0,1)^{n+1}$ such that $\sum_{j = 0}^n \lambda_j (\left(\alpha_j\right)_i - \beta_i) = 0$ for all $i \in [n]$, since $\supp(f)$ is a simplex which contains the origin in its interior.
	It follows that the minimizer $\xb_\ast$ must satisfy $c_je^{\langle \xb_{\ast}, \alpb_j - \betab \rangle} = \lambda_j$ for all $j = 0, 1, \ldots, n$.
	Consequently, we have that $\xb_\ast = \eq{f}$ if and only if 
	\begin{align*}
		|c_{\betab}| \ = \ c_j e^{\langle \xb, \alpb_j - \betab \rangle} \ = \  \lambda_j \qquad \text{ for all } j = 0, 1, \ldots, n.
	\end{align*}
	This can be satisfied if and only if $|c_{\betab}| = \lambda_j = \frac{1}{n+1}$ for all $j = 0, 1, \ldots, n$, i.e., if $\supp(f)$ is a barycentric circuit with $|c_{\betab}| = \frac{1}{n+1}$.
\end{proof}

We close the section by showing how the DSONC cone can be expressed in terms of the toric morphism $\varphi_A$ and the action $G$ introduced in the beginning of the section.

\begin{theorem}
	Let $A = A^+ \cup \{\betab\}$ be a circuit with $A^+ = \set{\alpb_0, \ldots, \alpb_n}$ and $\betab = \Vector{0}$, and barycentric coordinates $\Vector{\lambda} \in (0,1)^{n+1}$.
	Let $(\Vector{1},-1) \in \R^{n+2}$ denote the all-1-vector with a single negative sign in the last entry.
	Then we have
	\begin{align*}
		\circuitDSONC \ = \ \{ \langle G(\wb,\varphi_A(\x)), t \cdot (\Vector{1},-1) \rangle \in \R^A \ : \ \wb \in \R^n, \ t \in \R_{>0}\}.
	\end{align*}
	Hereby, the action $G$ moves the equilibrium point from the origin to an arbitrary point in $\R^n$ and adjusts the coefficients such that equation \cref{Equation:Equilibrium} still holds.
\label{theorem:EquilibraTorusAction}
\end{theorem}

Analogously to thinking about circuit functions in the SONC cone as \emph{``agiforms plus group action''}, we can thus think about those in the DSONC cone as a \emph{``circuit functions with coefficient vector $(\Vector{1},-1)$ plus group action''}.

\begin{proof}
	First, we show the direction ``$\subseteq$''. 
	As discussed in \cref{example:Motzkin}, circuit functions with coefficient vector $(\Vector{1},-1)$ are contained in the boundary of the DSONC cone with equilibrium point at the origin, which means that $\langle \varphi_A(\x), (\Vector{1},-1) \rangle\ \in \circuitDSONC$.
	Since furthermore
	\begin{align*}
		\langle G(\wb,\varphi_A(\x)), t\cdot (\Vector{1},-1) \rangle
		\ & = \  t\cdot \left(\sum_{j = 0}^n e^{\langle \xb - \wb, \alpb_j \rangle} - e^{\langle \xb - \wb, \betab \rangle}\right) \\
		& = \ t\cdot \left(\sum_{j = 0}^n e^{\langle -\wb, \alpb_j \rangle} \cdot e^{\langle \xb, \alpb_j \rangle} - \prod_{j = 0}^n \left(e^{\langle -\wb, \alpb_j \rangle}\right)^{\lambda_j} \cdot \expbeta\right),
	\end{align*}
	where the equation for the last step follows since $\Vector{\lambda}$ yields the barycentric coordinates of $\betab$ in terms of the $\alpb_j$, we have that $G(\wb,\varphi_A(\x)),  t\cdot (\Vector{1},-1) \rangle\ \in \circuitDSONC$.
	
	For the direction ``$\supseteq$'' we observe that every circuit function $f(\x)$ in $\partial \circuitDSONC$ has an arbitrary but unique equilibrium point $\eq{f} \in \R^n$ satisfying \cref{Equation:Equilibrium}.
	If $\eq{f}$ is fixed, then the coefficients of $f(\x)$ are uniquely determined up to multiplication with a positive scalar $t$.
	Since we have seen that the image of $G(\wb,\varphi_A(\x)), t\cdot (\Vector{1},-1) \rangle$ yields a circuit function in $\circuitDSONC$ with equilibrium point $\wb$, which is chosen arbitrarily in $\R^n$, the statement follows.
\end{proof}

\bibliographystyle{amsalpha}
\bibliography{closure_sonc_dual_arxiv_1.bib}

\appendix
\section{}
\label{section:Appendix}

We present proofs here which have been omitted in the main part of this paper for sake of brevity and readability. \\ 

In order to show \cref{proposition:ExtremeRays}, we need the following variant of H\"older's inequality.
\begin{theorem}[\cite{Hardy:Littlewood:Polya:Inequalities}]
	Let $m,n \in \N$ and let $(a_{ij}) \in \R^{n \times m}$.
	Let further $\Vector{\lambda} \in [0,1]^{n}$ satisfy $\sum_{i = 1}^n \lambda_i = 1$.
	It holds that 
	\begin{align}
		\sum_{j = 1}^m \prod_{i = 1}^n (a_{ij})^{\lambda_i} \ \le \ \prod_{i = 1}^n \left( \sum_{j = 1}^m a_{ij} \right)^{\lambda_i}.
	\label{eqn:Hoelder}
	\end{align}
	The inequality \cref{eqn:Hoelder} holds with equality if either there exists some $k \in [n]$ such that $a_{kj} = 0$ for all $j \in [m]$ or if the matrix $(a_{ij})$ has rank $1$.
\end{theorem}

\begin{proof}[Proof of \cref{proposition:ExtremeRays}]
	We need to show that 

	\begin{enumerate}
		\item Every function $f$ in $\ADSONC$ can be written as a sum of functions in $E_1$ and $E_2$.
		\item Functions in $E_1 \cup E_2$ cannot be written as sums of other elements in $\ADSONC$.
	\end{enumerate}

	\medskip

	We begin by showing part (1).
	It is clear from \cref{theorem:ReducedDualContainment} that every function in $\ADSONC$ can be written as a sum of minimal circuit functions $f^{(\betab)}$ in the DSONC cone, which are supported on some $C^+ \cup \set{\betab} \in \cCmin(A)$.
	As usual, $\betab \in A$ denotes the single inner point of $f^{(\betab)}$.
	If we write $f^{(\betab)} = \sum_{\alpb \in C^+} c^{(\betab)}_{\alpb} \expalpha + c^{(\betab)}_{\betab} \expbeta \in \mincircuitDSONC$, then by \cref{corollary:DualCircuitNumber} it holds that $-c^{(\betab)}_{\betab} \le \prod_{\alpb \in C^+} \left(c^{(\betab)}_{\alpb}\right)^{\lambda_{\alpb}}$ for the unique vector of barycentric coordinates $\Vector{\lambda} \in \Lambda(C^+, \betab)$.
	We can thus assume without loss of generality that $f$ is a minimal circuit function supported on $C^+ \cup \set{\betab} \in \cCmin(A)$ satisfying
	\begin{align*}
		-c_{\betab} \ \le \ \prod_{\alpb \in C^+} c_{\alpb}^{\lambda_{\alpb}} \ = \ \check\Theta_{f}.
	\end{align*}
	If $c_{\betab} \ge 0$, then $f$ is a sum of nonnegative exponential monomials and can thus trivially be written as a sum of functions in $E_2$.
	Assume now that $c_{\betab} < 0$.
	Since we have $-c_{\betab} = |c_{\betab}| \le \check\Theta_f$, we can write $f$ as a convex combination of functions $f_1$ and $f_2$ supported on (subsets of) $C^+ \cup \set{\betab}$, such that the coefficients of $f_1$ and $f_2$ corresponding to $C^+$ are identical to the positive coefficients of $f$, and the coefficients of $f_1$ and $f_2$ corresponding to $\betab$ are $-\check\Theta_f$ and $0$, respectively.
	Thus, $f_1 \in E_1$ and $f_2 \in E_2$.

	To prove part (2), assume that for $f \in E_1 \cup E_2$ there exist some {\AGE}s $g_{\betab} \in \circuitDSONC$, $\betab \in A^-$, such that $f = \sum_{\betab \in A^-} g_{\betab}$.
	First, we show that 
	\begin{align}
		\conv\left(\bigcup_{\betab \in A^-} \supp(g_{\betab})\right) \ \subseteq \ \conv\left(\supp(f)\right).
	\label{eqn:SupportContainment}
	\end{align}
	Let $\alpb \in \vertices{\conv\left(\bigcup_{\betab \in A^-} \supp(g_{\betab}) \right)}$.
	It follows that $\alpb \in \vertices{\conv\left(\supp(g_{\betab})\right)}$ for all $\betab$ such that $\alpb \in \supp(g_{\betab})$.
	Since all $g_{\betab}$ are nonnegative, it needs to hold for every $g_{\betab}$ that the coefficient $c^{(\betab)}_{\alpb}$ corresponding to $\alpb$ is nonnegative.
	Thus, $\sum_{\betab \in A^-} c^{(\betab)}_{\alpb} > 0$ and $\alpb \in \supp(f)$.
	The claim \cref{eqn:SupportContainment} follows.

	Assume now that $f \in E_1$.
	We denote the support set of $f$ by $C^+ \cup \set{\Vector{\gamma}}$.
	Then by \cref{eqn:SupportContainment}, all $g_{\betab}$ are of the form
	\begin{align*}
		g_{\betab} \ = \ \sum_{\alpb \in C^+} c^{(\betab)}_{\alpb} \expalpha + c^{(\betab)}_{\Vector{\gamma}} e^{\langle \xb, \Vector{\gamma} \rangle},
	\end{align*}
	where $c^{(\betab)}_{\alpb} \ge 0$ for all $\alpb \in C^+$.
	Since by assumption $g_{\betab} \in \circuitDSONC$, it holds that $|c^{(\betab)}_{\Vector{\gamma}}| \le \check\Theta_{g_{\betab}}$.
	It follows that 
	\begin{align}
		\begin{split}
		|c_{\Vector{\gamma}}| & \ = \ \sum_{\betab \in A^-} -c^{(\betab)}_{\Vector{\gamma}} 
		\ \le \ \sum_{\betab \in A^-} -\check\Theta_{g_{\betab}} 
		\ = \ \sum_{\betab \in A^-} \left( \prod_{\alpb \in C^+} \left(c^{(\betab)}_{\alpb} \right)^{\lambda_{\alpb}}\right) \\
		& \ \stackrel{(*)}{\le} \ \prod_{\alpb \in C^+} \left( \sum_{\betab \in A^-} c^{(\betab)}_{\alpb} \right)^{\lambda_{\alpb}} 
		\ = \ \prod_{\alpb \in C^+} c_{\alpb}^{\lambda_{\alpb}} 
		\ = \ \check\Theta_f 
		\ = \ |c_{\Vector{\gamma}}|,
		\end{split}
	\label{eqn:HoelderChain}
	\end{align}
	where $\Vector{\lambda}$ is the unique vector in $\Lambda(C^+, \Vector{\gamma})$. 
	Note that we have used the H\"older inequality in $(*)$.
	Since in this case the inequality holds with equality, we have that either there exists some $\alpb \in C^+$ such that $c^{(\betab)}_{\alpb} = 0$ for all $\betab \in A^-$ or the matrix $\left(c^{(\betab)}_{\alpb}\right)_{\betab, \alpb}$ has rank $1$.
	In the first case it follows that $c_{\Vector{\gamma}} = 0$, so we can disregard this.
	Consider the second case. 
	For the matrix $\left(c^{(\betab)}_{\alpb}\right)_{\betab, \alpb}$ to have rank $1$, there must exist some $\tilde{\Vector{c}} \in \R^{A^-}_{\ge 0}$ such that $c^{(\betab)}_{\alpb} = \tilde{c}^{(\betab)} c_{\alpb}$ for all $\betab \in A^-$ and all $\alpb \in C^+$.
	It follows from \cref{eqn:HoelderChain} that 
	\begin{align*}
		\left|c^{(\betab)}_{\Vector{\gamma}}\right| & \ = \ \prod_{\alpb \in C^+} (c^{(\betab)}_{\alpb})^{\lambda_{\alpb}} 
		\ = \ \prod_{\alpb \in C^+} (\tilde{c}^{(\betab)} c_{\alpb})^{\lambda_{\alpb}} 
		\ = \ \left(\prod_{\alpb \in C^+} (\tilde{c}^{(\betab)})^{\lambda_{\alpb}}\right) \cdot \check{\Theta}_f 
		\ = \ \tilde{c}^{(\betab)} \cdot |c_{\Vector{\gamma}}|.
	\end{align*}
	Thus, all $g_{\betab}$ are multiples of $f$ and the claim follows.
\end{proof}

\bigskip

We now show that \cref{proposition:ClosureAffine} holds, i.e., that the DSONC cone is closed under affine transformation of variables.

\begin{proof}[Proof of \cref{proposition:ClosureAffine}]
	Let $A = A^+ \cup \set{\betab} \subset \R^n$ denote the support set of $f$.
	Then $A$ is a circuit and $f$ is of the form
	\begin{align*}
		f(\xb) \ = \ \sum_{{\alpb}\in A^+} c_{\alpb} \expalpha + c_{\betab} \expbeta \ \in \ \circuitDSONC.
	\end{align*}
	Since
	\begin{align*}
		f(T(\xb)) = \sum_{{\alpb}\in A^+} c_{\alpb} e^{\langle \Vector{M} \xb + \Vector{a},{\alpb} \rangle} + c_{\betab} e^{\langle \Vector{M} \xb + \Vector{a},{\betab} \rangle} = \sum_{{\alpb}\in A^+} c_{\alpb} e^{\langle \Vector{a},{\alpb} \rangle}  e^{\langle \xb ,  \Vector{M}\T \alpb \rangle} + c_{\betab} e^{\langle \Vector{a}, \betab \rangle} e^{\langle \xb, \Vector{M}\T \betab \rangle},
	\end{align*}
	the support set of $f(T(\xb))$ is given by $\tilde{A}^+ \cup \set{\tilde{\betab}}$, where $\tilde{A}^+ \ = \ \set{\Vector{M}\T \alpb  \ : \  \alpb \in A^+}$ and $\tilde{\betab} = \Vector{M}\T \betab$.
	Note that $\tilde{A}^+ \cup \set{\tilde{\betab}}$ is still a circuit since
	\begin{align*}
		\tilde{\betab} \ = \ \Vector{M}\T \betab \ = \ \Vector{M}\T \sum_{\alpb \in A^+} \lambda_{\alpb} \alpb \ = \ \sum_{\alpb \in A^+} \lambda_{\alpb} \Vector{M}\T \alpb \ = \ \sum_{\tilde{\alpb} \in \tilde{A}^+} \lambda_{\tilde{\alpb}} \tilde{\alpb},
	\end{align*}
	where $\tilde{\alpb} = \Vector{M}\T \alpb$ and $\lambda_{\tilde{\alpb}} = \lambda_{\alpb}$ for all $\alpb \in A^+$.
	Since $f \in \circuitDSONC$, it holds that 
	\begin{align*}
		|c_{\betab}| \ \le \ \prod_{\alpb \in A^+} c_{\alpb}^{\lambda_{\alpb}}.
	\end{align*}
	Combining this with the previous observations yields 
	\begin{align*}
		\left|c_{\betab} \cdot e^{\langle \Vector{a}, \betab \rangle} \right| 
		\ &\le \ \left( \prod_{\alpb \in A^+} c_{\alpb}^{\lambda_{\alpb}} \right) \cdot e^{\langle \Vector{a}, \betab \rangle} 
		\ = \ \left( \prod_{\alpb \in A^+} c_{\alpb}^{\lambda_{\alpb}} \right) \cdot e^{\langle \Vector{a}, \sum_{\alpb \in A^+} \lambda_{\alpb}\alpb \rangle} \\
		\ &= \ \left( \prod_{\alpb \in A^+} c_{\alpb}^{\lambda_{\alpb}} \right) \cdot \left( \prod_{\alpb \in A^+} \left( e^{\langle \Vector{a}, \alpb \rangle}\right)^{\lambda_{\alpb}} \right)
		\ = \ \prod_{\alpb \in A^+} \left( c_{\alpb} \cdot e^{\langle \Vector{a}, \alpb \rangle}\right)^{\lambda_{\tilde\alpb}} .
	\end{align*}
	The claim now follows from \cref{corollary:DualCircuitNumber}.
\end{proof}

\end{document}